\documentclass{amsart} 
\usepackage{amssymb,amsmath,latexsym,times,xcolor,hyperref,tikz,enumitem}

\hypersetup{colorlinks=true, linkcolor=blue, citecolor=blue}

\numberwithin{equation}{section}

\usepackage{color}

\theoremstyle{plain}
\newtheorem{thm}{Theorem}[section]
\newtheorem{dfn}[thm]{Definition}
\newtheorem{cor}[thm]{Corollary}

\newtheorem{lemma}[thm]{Lemma}

\newtheorem*{lem*}{Lemma}
\newtheorem*{thm*}{Theorem}
\newtheorem*{cor*}{Corollary}

\theoremstyle{definition}
\newtheorem{example}{Example}

\newcommand{\del}{\backslash}
\DeclareMathOperator{\cl}{cl}

\DeclareMathOperator{\sort}{sort}

\title[A characterization of positroids]{A characterization of
  positroids, with applications to amalgams and excluded minors}

\author{Joseph E.~Bonin} \address{Department of Mathematics\\ The
  George Washington University\\ Washington, D.C. 20052, USA}
\email{jbonin@gwu.edu}

\keywords{positroid, base-sortable matroid, amalgam, excluded minor}

\date{\today.}

\begin{document}

\maketitle

\begin{abstract}
  A matroid of rank $r$ on $n$ elements is a positroid if it has a
  representation by an $r$ by $n$ matrix over $\mathbb{R}$, each $r$
  by $r$ submatrix of which has nonnegative determinant.  Earlier
  characterizations of connected positroids and results about direct
  sums of positroids involve connected flats and non-crossing
  partitions.  We prove another characterization of positroids of a
  similar flavor and give some applications of the characterization.
  We show that if $M$ and $N$ are positroids and the intersection of
  their ground sets is an independent set and a set of clones in both
  $M$ and $N$, then the free amalgam of $M$ and $N$ is a positroid,
  and we prove a second result of that type.  Also, we identify
  several multi-parameter infinite families of excluded minors for the
  class of positroids.
\end{abstract}

\section{Introduction}\label{sec:intro}

A \emph{positroid} is a matroid $M$ on a set $E(M)$, say of rank $r$
and with $|E(M)|=n$, that has a representation by an $r$ by $n$ matrix
over $\mathbb{R}$, each $r$ by $r$ submatrix of which has nonnegative
determinant.  Since Blum \cite{blum} (as base-orderable matroids) and
Postnikov \cite{post} introduced positroids, many papers have
developed this topic, spurred by the important connections with other
branches of mathematics that led to the introduction of positroids as
well as the wealth of lenses through which they can be viewed (e.g.,
Grassmann necklaces, reduced plabic graphs, decorated permutations).

Positroids are often defined to have a fixed linear order on the
ground set, as is inherent in the columns of a matrix (and so $E(M)$
is often taken to be $[n]=\{1,2,\ldots,n\}$).  With our focus on
matroids per se, we find it useful to separate the matroid and the
linear order, so in this paper, a positroid is a matroid $M$ for which
there is a linear order on the ground set $E(M)$, which we call a
\emph{positroid order} for $M$, so that the elements, when put in that
order, correspond to the columns of a matrix representation of the
type described above.

Blum \cite{blum} and Ardila, Rinc\'on, and Williams \cite{ARW} treat
some of the fundamental properties of the class of positroids.  In
particular, they show that the class of positroids is closed under the
operations of minors, duals, and direct sums.  Thus, a matroid is a
positroid if and only if the restrictions to its connected components
are positroids.  This raises the issue of determining how to combine
the positroid orders on the connected components to get a positroid
order on the direct sum.  Ardila, Rinc\'on, and Williams settled this
question using the well-studied notion of a \emph{non-crossing
  partition}, that is, a partition $\{X_1,X_2,\ldots,X_k\}$ of a
linearly ordered set $E$ for which, for any two blocks $X_i$ and
$X_j$, there is a subset $A$ of $E$ such that $X_i\subseteq A$,
$X_j\subseteq E- A$, and either $A$ or $E-A$ is an interval in the
linear order.  They proved the following result \cite[Theorem
7.6]{ARW}.

\begin{thm}\label{thm:arw1}
  Let $X_1,X_2,\ldots,X_k$ be the connected components of a matroid
  $M$.  For any positroid order for $M$, the partition
  $\{X_i\,:\,i\in[k]\}$ of $E(M)$ is non-crossing.  Conversely, any
  linear order on $E(M)$ for which the induced order on each $X_i$ is
  a positroid order for $M|X_i$ and the partition
  $\{X_i\,:\,i\in[k]\}$ of $E(M)$ is non-crossing is a positroid order
  for $M$.
\end{thm}

That result along with the following characterization of positroid
orders for connected matroids \cite[Proposition 5.6]{ARW} give a
two-part characterization of positroid orders: Theorem \ref{thm:arw2}
characterizes the orders on the connected components while Theorem
\ref{thm:arw1} characterizes how the orders on those components are
assembled globally.

\begin{thm}\label{thm:arw2}
  Let $M$ be a connected matroid with $r(M)\geq 2$.  A linear order
  $\leq$ on $E(M)$ is a positroid order for $M$ if and only if, for
  every flat $F$ such that both $M|F$ and $M/F$ are connected, at
  least one of $F$ and $E(M)-F$ is an interval in $\leq$.
\end{thm}

The following similar characterization of positroid orders for
connected matroids was given by Rinc\'{o}n, Vinzant, and Yu
\cite[Proposition 6.5]{pchar}, and likewise can be paired with Theorem
\ref{thm:arw1} to characterize positroids.

\begin{thm}\label{thm:pchar}
  Let $M$ be a rank-$r$, connected matroid.  Fix an integer $k$ with
  $1 < k \leq r$.  A linear order on $E(M)$ is a positroid order for
  $M$ if and only if, for each flag
  $$\emptyset=F_0\subsetneq F_1 \subsetneq F_2 \subsetneq\cdots
  \subsetneq F_{k-1}\subsetneq F_k= E(M)$$ of flats of $M$ for which
  each minor $M|F_i/F_{i-1}$, for $i\in[k]$, is connected, the set of
  their ground sets, $\{F_i-F_{i-1}\,:\,i\in[k]\}$, is a non-crossing
  partition of $E(M)$.
\end{thm}

In Theorem \ref{thm:char}, we give a similar characterization of
positroids, namely, the following result, which, along with its
corollaries, the rest of the paper exploits.  \begingroup
\addtolength\leftmargini{-0.75cm}
\begin{quote}
  \emph{Let $M$ be a matroid with no loops.  A linear order $\leq$ on
    $E(M)$ is a positroid order if and only if, for each connected
    flat $F$ of $M$, if $K$ is a connected component of the
    contraction $M/F$, then there is a subset $I$ of $E(M)$ such that
    $F\subseteq E(M)-I$, $K\subseteq I$, and either $I$ or $E(M)-I$ is
    an interval in $\leq$.}
\end{quote}
\endgroup

\noindent
This result applies to all loopless matroids, whether connected or
not, thereby making it easier to use.  There is much interest in
characterizations of positroid orders; for instance, see Agarwala,
Delaney, Yeats \cite[Remark 5.27]{susama}, which relates Theorem 5.17
in that paper to Theorem \ref{thm:char} in this paper.  In Section
\ref{sec:char} we also treat some corollaries of Theorem
\ref{thm:char} and, in Theorem~\ref{thm:equivchars}, relate it to
several earlier characterizations of positroids.

We develop more extensive applications of Theorem \ref{thm:char} in
Sections \ref{sec:bonding} and \ref{sec:exmin}, the first of which
concerns free amalgams.  The free amalgam of matroids $M$ and $N$,
which exists only for certain pairs $(M,N)$, is the freest matroid on
$E(M)\cup E(N)$ that has $M$ and $N$ as restrictions.  A subset $X$ of
$E(M)$ is a set of clones if each permutation of $E(M)$ that fixes
each element of $E(M)-X$ is an automorphism of $M$.  We show that if
$M$ and $N$ are positroids and $E(M)\cap E(N)$ is an independent set
of clones in both $M$ and $N$, then the free amalgam of $M$ and $N$
exists and is a positroid (Theorem \ref{thm:bondmain1}).  We also show
that the same conclusion holds under conditions on the independent set
$E(M)\cap E(N)$ that are in one way weaker and in another stronger
(Theorem \ref{thm:bondmain2}).  A contraction of the free amalgam of
$M$ and $N$ need not be the free amalgam of the corresponding
contractions of $M$ and $N$; because of this, we introduce and work
with a more general way to glue matroids together, one that has useful
contraction properties.  Defining and developing the basic properties
of this construction, which we call bonding, occupies much of Section
\ref{sec:bonding}.  Proving all relevant properties of bonding makes
the paper more self-contained; no knowledge of free amalgams is
required.  In Section \ref{sec:exmin}, we identify many excluded
minors for the class of positroids, including multi-parameter infinite
families of excluded minors.  We do not identify all excluded minors
for the class of positroids; for some excluded minors not given here,
see Park \cite{Park}.

In Section \ref{sec:background} we review the background that is
needed for these results.  To make the paper accessible to a wide
audience, we give complete background on, for instance, the orders
used when treating Grassmann necklaces, as well as sketches of some of
the less widely known topics in matroid theory.  We assume familiarity
with basic matroid theory, as set out in \cite{oxley}, and we follow
the matroid notation used in \cite{oxley}.

\section{Background}\label{sec:background}

We use $\mathbb{Z}$ for the set of integers, $\mathbb{N}$ for
$\{n\in\mathbb{Z}\,:\,n>0\}$, and $[n]$ for $\{1,2,\ldots,n\}$.

\subsection{Some orders related to a linear order}
For a linear order $\leq$ on a set $E$, say given by
$e_1<e_2<\cdots< e_n$, its \emph{$i$-shift} $\leq_i$ is the cyclic
shift in which $e_i$ is least, that is,
$$e_i<_i e_{i+1}<_i\cdots <_i e_n <_i e_1<_i \cdots<_i e_{i-1}.$$
A \emph{cyclic interval} of $\leq$ is an interval in some $i$-shift of
$\leq$.  Elements $e$ and $f$ of $E$ are \emph{cyclically consecutive}
in $\leq$ if $\{e,f\}$ is a cyclic interval of $\leq$. In an ordered
set $(P,\leq)$, an \emph{ideal} is a subset $I$ of $P$ for which if
$x\in I$ and $y\in P$ with $y\leq x$, then $y\in I$; also,
$F\subseteq P$ is a \emph{filter} if whenever $x\in F$ and $z\in P$
with $x\leq z$, then $z\in F$.  For a linear order, an ideal is an
interval that contains the least element; a filter is an interval that
contains the greatest element.  Note that for $A\subseteq E$, the
following statements are equivalent to saying that $A$ is a cyclic
interval: $E-A$ is a cyclic interval; $A$ or $E-A$ is an interval; $A$
is an ideal of some shift $\leq_i$ (so $e_i$ is the least element of
$A$ using $\leq_i$); $A$ is a filter of some shift $\leq_j$ (so
$e_{j-1}$ is the greatest element of $A$ using $\leq_j$).

In some examples we will need to show that no linear order on $E$ has
all sets in some set of sets being cyclic intervals; for instance,
$\{1,3,5\}$, $\{2,3,4\}$, $\{4,5,6\}$, and $\{1,2,6\}$ cannot all be
cyclic intervals in any linear order on $[6]$.  For this, the
following observation and lemma are useful.  Given two cyclic
intervals $A$ and $B$ in a linear order on $E$, if
$A\cap B\ne\emptyset$, then either $A\cup B=E$ or $A\cap B$ is a
cyclic interval.
  
\begin{lemma}\label{lem:noco}
  Let $A$, $B$, $C$, and $D$ be cyclic intervals in a linear order on
  $E$ with $A\not\subseteq B$, $A\cup B\ne E$, $A\cap B\ne \emptyset$,
  $C\not\subseteq B$, $B\cup C\ne E$, $B\cap C\ne \emptyset$, and
  $A\cap B\cap C= \emptyset$.  If $D$ is disjoint from $A\cap B$ and
  $B\cap C$ but not from $B-(A\cup C)$, then $D\subseteq B-(A\cup C)$.
\end{lemma}

\begin{proof}
  By replacing the order by a cyclic shift if needed, we may assume
  that $B$ is an interval.  The result follows by observing that, by
  the hypotheses, one of $A\cap B$ and $B\cap C$ is an ideal in the
  order that is induced on $B$, and the other is a filter, so only one
  of the two cyclic intervals whose union is $E-(B\cap (A\cup C))$ can
  contain elements of $D$.
\end{proof}

For a rank-$r$ matroid $M$, if the matrix representation $A$ shows
that $\leq$ is a positroid order for $M$, then the matrix $A'$ that
shifts the first column of $A$ to be last and multiplies that column
by $(-1)^{r-1}$ shows that $\leq_2$ is also a positroid order for $M$
since one can shift the first column of an $r$ by $r$ matrix to the
last place with $r-1$ column transpositions.  Thus, $\leq$ is a
positroid order for a matroid $M$ if and only if $\leq_i$ is a
positroid order for $M$.

Given a linear order $\leq$ on an $n$-element set $E$ and an integer
$k\in[n]$, we consider two associated orders on the set of $k$-element
subsets of $E$: for $k$-subsets $X$ and $Y$ of $E$ where the elements
of $X$ are, in order, $x_1<x_2<\cdots <x_k$, and likewise
$y_1<y_2<\cdots <y_k$ for $Y$, we have
\begin{itemize}[leftmargin=*]
\item the \emph{Gale order}: $X\leq_G Y$ if $x_i\leq y_i$ for all
  $i\in [k]$, and
\item the \emph{lexicographic order}: $X\leq_L Y$ if $X=Y$ or
  $x_i<y_i$ for the least $i$ with $x_i\ne y_i$.
\end{itemize}
While $\leq_L$ is a linear order, $\leq_G$ is not except for certain
$n$ and $k$.  Also, $\leq_L$ is an extension of $\leq_G$, that is, if
$X\leq_G Y$, then $X\leq_L Y$.

By the next lemma, as one would expect, replacing an element in a set
by a smaller element yields a smaller set in the Gale order.

\begin{lemma}\label{lem:a<b}
  Let $\leq$ be a linear order on $E$. Fix a subset $X$ of $E$ and
  elements $a\in X$ and $b\in E-X$.  If $b<a$, then $(X-a)\cup b<_GX$.
  If $a<b$, then $X<_G(X-a)\cup b$.
\end{lemma}

\begin{proof}
  Assume that $b<a$. To get $(X-a)\cup b<_GX$, compare the elements of
  $X$, in order, (the first line below) to those of $(X-a)\cup b$ (the
  second line):
  $$x_1<\cdots< x_{j-1}<x_j<x_{j+1}\cdots <x_{i-1}<a<x_{i+1}<\cdots<x_k$$
  $$x_1<\cdots< x_{j-1}<b<x_j<\cdots<x_{i-2}<x_{i-1}<x_{i+1}<\cdots<x_k$$
  where $j$ is the least integer with $b<x_j$.  The second statement
  follows from the first.
\end{proof}

The next lemma is from Gale \cite{gale}.

\begin{lemma}\label{lem:GaleIsMin}
  Let $M$ be a matroid.  Let $\leq$ be a linear order on $E(M)$.  If
  $B$ is the least basis of $M$ in the lexicographic order, then, in
  the Gale order, $B\leq_G B'$ for all bases $B'$ of $M$.
\end{lemma}

\begin{proof}
  We induct on the distance between bases $B$ and $B'$ in the
  lexicographic order.  If $B'=B$, then $B\leq_G B'$, so we may assume
  that $B\ne B'$ and, inductively, that $B\leq_G A$ for all bases $A$
  for which $A<_LB'$.  Fix $a\in B'-B$.  By the symmetric exchange
  property, for some $b\in B-B'$, both $(B-b)\cup a$ and
  $(B'-a)\cup b$ are bases of $M$.  If $a<b$, then Lemma \ref{lem:a<b}
  would give $(B-b)\cup a<_GB$, so $(B-b)\cup a<_LB$, contrary to $B$
  being the least basis in the lexicographic order.  Thus, $b<a$, so
  Lemma \ref{lem:a<b} gives $(B'-a)\cup b<_GB'$.  By the inductive
  assumption, $B\leq_G (B'-a)\cup b$, so $B<_GB'$.
\end{proof}

The least basis of a matroid $M$ in the Gale order induced by a linear
order is the \emph{Gale basis} for that order.  (In some sources, the
Gale basis is taken to be the greatest basis.)

\subsection{Transversal and nested matroids}

Let $\mathcal{A}=(A_i\,:\,i\in[r])$ be an indexed family of subsets of
a set $E$.  A \emph{partial transversal} of $\mathcal{A}$ is a subset
$X$ of $E$ for which there is an injection $\phi:X\rightarrow [r]$
with $e\in A_{\phi(e)}$ for all $e\in X$.  Edmonds and
Fulkerson~\cite{ef} showed that the partial transversals of
$\mathcal{A}$ are the independent sets of a matroid on $E$.  This
matroid, denoted $M[\mathcal{A}]$, is a \emph{transversal matroid},
and $\mathcal{A}$ is a \emph{presentation} of it.  In \cite{aff},
Brylawski showed how to view transversal matroids geometrically:
transversal matroids of rank $r$ with no loops are precisely the
matroids that have geometric representations on an $r$-vertex simplex
in which, for all $k\in[r]$, each circuit of rank $k$ spans a
$k$-vertex face of the simplex (e.g., a $3$-element circuit must span
an edge of the simplex).

Given a linear order $\leq$ on $E$ and an $r$-subset $I$ of $E$, for
each $i\in I$, let $F_i$ be the filter $\{j\,:\,i\leq j\}$, and let
$\mathcal{A}=(F_i\,:\,i\in I)$.  Let $N(I,\leq)$ be the transversal
matroid $M[\mathcal{A}]$.  Note that if $i<k$, then
$F_k\subsetneq F_i$, so the presentation $\mathcal{A}$ of $N(I,\leq)$
is a chain of sets under inclusion.  Transversal matroids that have
presentations that are chains of sets are called \emph{nested
  matroids}. (They are also called shifted matroids and Schubert
matroids, among other names; see \cite[Section 4]{lpm2}.)  Note that
$I$ is the Gale basis of $N(I,\leq)$, and an $r$-subset $J$ of $E$ is a
basis of $N(I,\leq)$ if and only if $I\leq_G J$.

\subsection{Grassmann necklaces and base-sortable matroids}

For a matroid $M$ and linear order $e_1<e_2<\cdots<e_n$ on $E(M)$, if
$I_i$ is the Gale basis of $M$ using the $i$-shift $\leq_i$ of $\leq$,
then the $n$-tuple $(I_1,I_2,\ldots,I_n)$ of bases has the properties
\begin{itemize}[leftmargin=*]
\item if $e_i\in I_i$, then $I_{i+1} = (I_i-e_i)\cup e_j$ for some
  $j\in [n]$ ($j$ can be $i$), and
\item if $e_i\not\in I_i$, then $I_{i+1} = I_i$,
\end{itemize}
where $I_{n+1}=I_1$.  An $n$-tuple of $r$-element subsets of a
linearly ordered $n$-element set with these properties is a
\emph{Grassmann necklace of type $(n,r)$.}  The next result combines a
result of Postnikov \cite{post} (the first part) with one of Oh
\cite{oh} (the converse).  In this result and the rest of this paper,
$\mathcal{B}(M)$ denotes the set of bases of a matroid $M$.

\begin{thm}\label{thm:postoh}
  Let $\leq$ be a linear order on $E$.  If $(I_1,I_2,\ldots,I_n)$ is a
  Grassmann necklace of type $(n,r)$ on $E$, then
  \begin{equation}\label{eq:posordereqn}
    \bigcap_{i\in[n]}\mathcal{B}\bigl(N(I_i,\leq_i)\bigr)
  \end{equation}
  is the set of bases of a rank-$r$ positroid on $E$ and $\leq$ is a
  positroid order for this matroid.  Conversely, if $\leq$ is a
  positroid order for a matroid $M$, then $\mathcal{B}(M)$ is the
  intersection \emph{(\ref{eq:posordereqn})} where $I_i$ is the Gale
  basis of $M$ for $\leq_i$.
\end{thm}

Positroids first appeared, under the name of base-sortable matroids,
in Blum \cite{blum}.  Let $\leq$ be a linear order on the ground set
$E(M)$ of a matroid $M$ of rank $r$.  Given bases $B$ and $B'$ of $M$,
list the elements of the multiset union $B\cup B'$ in order, taking
multiplicities into account, as $x_1\leq x_2\leq \cdots \leq x_{2r}$.
Let $B_e=\{x_{2i}\,:\,i\in [r]\}$, the set of elements in the even
positions in the list, and $B_o=\{x_{2i-1}\,:\,i\in [r]\}$, the set of
elements in the odd positions.  Set $\sort_\leq(B,B')=\{B_e,B_o\}$.
Thus, $\sort_\leq(B,B')= \sort_\leq(B',B)$.  Under the order dual
$\leq^d$ of $\leq$ (i.e., $e\leq^d f$ if and only if $f\leq e$), the
elements in the even positions are switched with those in the odd
positions, so $\sort_{\leq^d}(B,B') =\sort_{\leq}(B,B')$.  Likewise,
for the $i$-shift, $\sort_{\leq_i}(B,B')=\sort_\leq(B,B')$.
 
\begin{dfn}
  A matroid $M$ is \emph{base-sortable} if there is a linear order on
  $E(M)$ for which, for every pair $B$, $B'$ of bases of $M$, the sets
  $B_e$ and $B_o$ in $\sort(B,B')$ are bases of $M$.  Such a linear
  order on $E(M)$ is called a \emph{base-sorting order} for $M$.
\end{dfn}

Base-sortable matroids are also often called \emph{sort-closed
  matroids}.  As Blum noted, the observations above about $i$-shifts
and the dual order show that $\leq$ is a base-sorting order for $M$ if
and only if $\leq_i$ is a base-sorting order for $M$, and likewise for
the dual order.

Lam and Postnikov \cite[Corollary 9.4]{LamPost} proved that positroids
and base-sortable matroids are the same.  (Key results are in
\cite[Section 4.2]{LamPostI}.)

\begin{thm}\label{thm:bs=p}
  Let $M$ be a matroid.  A linear order on $E(M)$ is a positroid order
  if and only if it is a base-sorting order.
\end{thm}

\subsection{Connectivity}

For a matroid $M$, let $e,f\in E(M)$ be related precisely when $e=f$
or there is a circuit $C$ of $M$ with $\{e,f\}\subseteq C$.  This
relation is an equivalence relation and the equivalence classes are
the \emph{connected components} of $M$ (see \cite[Proposition
4.1.2]{oxley}).  A matroid $M$ is \emph{connected} if it has just one
connected component; equivalently, for each pair of distinct elements
$e,f\in E(M)$, there is a circuit $C$ of $M$ with
$\{e,f\}\subseteq C$.  A \emph{separator} of $M$ is a union of
connected components of $M$.  The \emph{trivial separators} are
$\emptyset$ and $E(M)$.  For a partition $\{X_1,X_2,\ldots,X_k\}$ of
$E(M)$ with $k>1$, we have
\begin{equation}\label{eq:dsdecomp}
  M=(M|X_1)\oplus(M|X_2)\oplus\cdots\oplus (M|X_k)
\end{equation}
if and only if $M$ and the direct sum have the same circuits, and that
holds if and only if each set $X_i$ is a nontrivial separator of $M$.
Since taking direct sums and the dual commute, it follows that a
matroid and its dual have the same connected components, so $M$ is
connected if and only if for each pair of distinct elements
$e,f\in E(M)$, there is a cocircuit $C^*$ of $M$ with
$\{e,f\}\subseteq C^*$.

\begin{lemma}\label{lem:dsdecomprank}
  For a matroid $M$ and a partition $\{X_1,X_2,\ldots,X_k\}$ of
  $E(M)$, Equation (\ref{eq:dsdecomp}) holds if and only if
  $r(M)=r(X_1)+r(X_2)+\cdots+r(X_k)$.
\end{lemma}

\begin{proof}
  Equation (\ref{eq:dsdecomp}) clearly gives
  $r(M)=r(X_1)+\cdots+r(X_k)$.  For the converse, it suffices to prove
  that if the rank equality holds and $C$ is a circuit of $M$, then
  $C\subseteq X_i$ for some $i\in [k]$.  Assume that this is false.
  Set $Y_i=X_i\cap C$ for $i\in [k]$. By our assumption, each set
  $Y_i$ is independent.  Let $B_i$ be a basis of $M|X_i$ with
  $Y_i\subseteq B_i$.  This gives the contradiction that
  $B_1\cup \cdots\cup B_k$ spans $M$, has size $r(M)$, but contains
  the circuit $C$.
\end{proof}

We say that a set $X$ in a matroid $M$ is \emph{connected} if the
restriction $M|X$ is connected.

\begin{lemma}\label{lem:connflcontr}
  Let $F$ be a nonempty connected flat of a matroid $M$.  If $A$ is a
  nonempty connected flat of the contraction $M/F$, then either $A$ or
  $A\cup F$ is a connected flat of $M$.
\end{lemma}

\begin{proof}
  Now $A\cup F$ is a flat of $M$ since $A$ is a flat of $M/F$.  Let
  $M|(A\cup F)$ be the direct sum
  $M_1\oplus M_2\oplus \cdots \oplus M_k$ where each $M_i$ is
  connected.  Since $F$ is connected in $M$, we may assume that
  $F\subseteq E(M_1)$.  If $F= E(M_1)$, then
  $M/F|A=M_2\oplus \cdots \oplus M_k$, so, since $A$ is a connected
  flat of $M/F$, we must have $k=2$, and therefore $A$ is a connected
  flat of $M$.  If $F\subsetneq E(M_1)$, then, since
  $M/F|A=(M_1/F)\oplus M_2\oplus \cdots \oplus M_k$ and $A$ is a
  connected flat of $M/F$, we must have $k=1$, so the flat $A\cup F$
  of $M$ is connected.
\end{proof}

The simple properties in the next lemma will come up at several points.

\begin{lemma}\label{lem:cocofcon}
  All circuits of a restriction of a matroid $M$ are circuits of
  $M$. Dually, all cocircuits of a contraction of $M$ are cocircuits
  of $M$.  In particular, coloops of contractions of $M$ are coloops
  of $M$.
\end{lemma}

Recall that the closure operator $\cl$ of a matroid $M$ can be cast in
terms of circuits as follows: for $X\subseteq E(M)$, we have
$\cl(X)=X\cup \{e\,:\, Y\cup e \text{ is a circuit for some }
Y\subseteq X\}$.

\begin{lemma}\label{lem:conncond}
  Let $M$ be a matroid with no loops.
  \begin{itemize}
  \item[\emph{(1)}] If $C$ is a circuit of $M$, then $\cl(C)$ is
    connected.
  \item[\emph{(2)}] If $M$ has a spanning circuit $C$, i.e.,
    $r(C)=r(M)$, then $M$ is connected.
  \item[\emph{(3)}] If some $e\in E(M)$ is not a coloop and is in no
    proper connected nonsingleton flat of $M$, then $M$ is connected,
    as is $M/F$ for any flat $F$ of $M$ with $e\not\in F$.
  \end{itemize}
\end{lemma}

\begin{proof}
  Part (1) follows from two items recalled above, the equivalence
  relation, applied to $M|\cl(C)$, and the formulation of $\cl(C)$.
  Part (2) follows from that since $\cl(C)=E(M)$.  For part (3), since
  $e$ is not a coloop, $e$ is in at least one circuit, say $C$, and
  $C$ spans $M$ since $\cl(C)$ cannot be a proper flat, so $M$ is
  connected by part (2).  For the second assertion, it suffices to
  show that the hypotheses of part (3) hold for $M/F$.  Now $e$ is not
  a coloop of $M/F$ by Lemma \ref{lem:cocofcon}.  Let the connected
  components of $M|F$ be $F_1,F_2,\ldots,F_t$.  No proper connected
  nonsingleton flat of $M/F_1$ contains $e$ since otherwise $e$ would
  be in a proper connected nonsingleton flat of $M$ by Lemma
  \ref{lem:connflcontr}.  The same argument, using $M/F_1$, then shows
  that $e$ is in no proper connected nonsingleton flat of
  $M/(F_1\cup F_2)$.  Iterating this show that, as needed, $e$ is in
  no proper connected nonsingleton flat of $M/F$.
\end{proof}
  
\subsection{Cyclic flats,  clones, and paving matroids}

Some of the excluded minors for the class of positroids that we
present in Section \ref{sec:exmin} will be defined by giving their
cyclic flats and the ranks of those flats, so we briefly review this
perspective on matroids.  Given a matroid $M$, a subset $A$ of $E(M)$
is \emph{cyclic} if the restriction $M|A$ has no coloops.
Equivalently, $A$ is cyclic if it is a (possibly empty) union of
circuits. A \emph{cyclic flat} is a flat that is cyclic.  Let
$\mathcal{Z}(M)$ denote the set of cyclic flats of $M$.  Note that
$\mathcal{Z}(M)$ is a lattice: the join of cyclic flats $A$ and $B$ is
$\cl(A\cup B)$, and their meet is $(A\cap B)-C$ where $C$ is the set
of coloops of $M|(A\cap B)$.  All connected flats $F$ with $|F|\geq 2$
are cyclic flats, but not conversely.

There is no connection between the terms \textsl{cyclic interval} and
\textsl{cyclic flat}, both of which are well established in the
literature.  For instance, using the usual order, in the cycle matroid
of the graph shown in Figure \ref{fig:cyclicvscyclic}, which is a
positroid, the set $\{3,6,9\}$ is a cyclic flat that is not a cyclic
interval, while $\{6,7\}$ is a cyclic interval that is a flat, but it
is not a cyclic flat.

\begin{figure}
  \centering
  \begin{tikzpicture}[scale=0.8]
    \tikzset{VertexV/.style = {circle, scale=0.9, fill=black!30,
        draw=black, thick}}

    \node[VertexV] (1) at (90:1) {};%
    \node[VertexV] (2) at (210:1) {};%
    \node[VertexV] (3) at (330:1) {};%
    \node[VertexV] (a) at (30:2) {};%
    \node[VertexV] (b) at (150:2) {};%
    \node[VertexV] (c) at (270:2) {};%

    \foreach \from/\to in {1/2,2/3,3/1,1/a,1/b,2/b,2/c,3/c,3/a}
    \draw[very thick](\from)--(\to);

    \node at (0,-0.75) {\small$9$};%
    \node at (0.65,0.3) {\small$6$};%
    \node at (-0.65,0.3) {\small$3$};%
    \node at (-1,1.2) {\small$2$};%
    \node at (1,1.2) {\small$4$};%
    \node at (-1.5,0.2) {\small$1$};%
    \node at (1.5,0.2) {\small$5$};%

     \node at (-0.6,-1.4) {\small$7$};%
     \node at (0.6,-1.4) {\small$8$};%

  \end{tikzpicture}
  \caption{A graph whose cycle matroid is a
    positroid.}\label{fig:cyclicvscyclic}
\end{figure}

For any set $X$ in a matroid $M$, if $F$ is the closure of the union
of the circuits of $M|X$, then $F$ is a cyclic flat of $M$ and
$r(X)=r(F)+|X-F|$ since the elements of $X-F$ are coloops of $M|X$.
If $A\subseteq E(M)$, then $r(X)\leq r(X\cap A)+|X-A|\leq r(A)+|X-A|$.
Thus, $r(X)=\min\{r(F)+|X-F|\,:\,F\in\mathcal{Z}(M)\}$.  So $M$ is
determined by its cyclic flats and their ranks, that is, by the set
$\{(A,r(A))\,:\,A\in \mathcal{Z}(M)\}$.  The next result, from
\cite{sims,cycflats}, provides an axiom scheme for matroids from this
perspective.

\begin{thm}\label{thm:cfaxioms}
  For a set $\mathcal{Z}$ of subsets of a set $E$ and a function
  $r : \mathcal{Z}\to \mathbb{Z}$, there is a matroid $M$ on $E$ with
  $\mathcal{Z}(M) = \mathcal{Z}$ and $r_M(X) = r(X)$ for all
  $X\in\mathcal{Z}$ if and only if
  \begin{itemize}
  \item[\emph{(Z0)}] $(\mathcal{Z},\subseteq)$ is a lattice,
  \item[\emph{(Z1)}] $r(0_\mathcal{Z})=0$, where $0_\mathcal{Z}$ is
    the least set in $\mathcal{Z}$,
  \item[\emph{(Z2)}] $0<r(Y)-r(X)<|Y-X|$ for all $X,Y \in \mathcal{Z}$
    with $X\subsetneq Y$, and
  \item[\emph{(Z3)}] for all sets $X, Y$ in $\mathcal{Z}$ (or,
    equivalently, just incomparable sets in $\mathcal{Z}$),
    \begin{equation}\label{ineq:submcyc}
      r(X\vee Y)+r(X\wedge Y)+|(X\cap Y)-(X\wedge Y)|\leq r(X)+r(Y).
    \end{equation}
  \end{itemize}
\end{thm}

A set $A$ is cyclic in a matroid $M$ if and only if $E(M)-A$ is a flat
of the dual matroid, $M^*$, since $A$ being a union of circuits of $M$
is equivalent to $E(M)-A$ being an intersection of hyperplanes
(cocircuit complements) of $M^*$.  Thus, $X$ is a cyclic flat of $M$
if and only if $E(M)-X$ is a cyclic flat of $M^*$, and so
$\mathcal{Z}(M^*)$, the lattice of cyclic flats of $M^*$, is
isomorphic to the order dual of $\mathcal{Z}(M)$.  In particular, just
as $\cl_M(\emptyset)$, the set of loops of $M$, is the least cyclic
flat of $M$, the set $E(M)-\cl_{M^*}(\emptyset)$ is the greatest
cyclic flat of $M$.  The first assertion in the lemma below is easy to
see directly, and the second follows by duality.

\begin{lemma}\label{lem:cyclicminor}
  Let $X$ be a cyclic flat of $M$.  The lattice $\mathcal{Z}(M|X)$ of
  cyclic flats of the restriction $M|X$ is the interval
  $[\cl(\emptyset),X]$ in $\mathcal{Z}(M)$.  The lattice
  $\mathcal{Z}(M/X)$ of cyclic flats of the contraction $M/X$ is
  $\{F-X\,:\,F \in\mathcal{Z}(M) \text{ and } X\subseteq F\}$, which
  is isomorphic to the upper interval $[X,E(M)-\cl_{M^*}(\emptyset)]$
  of cyclic flats that contain $X$ in $\mathcal{Z}(M)$.
\end{lemma}

From the remarks before Theorem \ref{thm:cfaxioms}, it follows that
the automorphisms of $M$ are the permutations of $E(M)$ that map each
cyclic flat of $M$ to a cyclic flat of the same rank.  Elements $e$
and $f$ in a matroid $M$ are \emph{clones of $M$} if the transposition
$(e,f)$, which maps $e$ to $f$, and $f$ to $e$, and fixes all other
elements of $E(M)$, is an automorphism of $M$. Since $M$ and $M^*$
have the same automorphisms, two elements are clones in a matroid if
and only if they are clones in its dual.  A subset $X$ of $E(M)$ is a
\emph{set of clones of $M$} if $e,f\in X$ are clones whenever
$e\ne f$. The next lemma holds since any permutation of $E(M)$ that
fixes each element that is not in a set $X$ of clones is a composition
of transpositions of clones.

\begin{lemma}\label{lem:autclone}
  For any matroid $M$ and set $X$ of clones of $M$, any permutation
  $\phi$ of $E(M)$ for which $\phi(e)=e$ for all $e\in E(M)-X$ is an
  automorphism of $M$,
\end{lemma}
  
Note that $e$ and $f$ are clones in $M$ if and only if they are in the
same cyclic flats of $M$, and, if $M$ has no loops, that is equivalent
to $e$ and $f$ being in the same connected flats $F$ of $M$ with
$|F|\geq 2$ since such flats, which are cyclic, are the connected
components of all cyclic flats.  Thus, if $X$ is a set of clones and
$F$ is a connected flat with $|F|\geq 2$, then either
$X\cap F=\emptyset$ or $X\subseteq F$.  For any $Y\subseteq E(M)$, the
closure in $M$ of any cyclic flat of $M|Y$ is a cyclic flat of $M$, so
clones in $M$ that are in $Y$ are clones in $M|Y$.  By duality, the
same holds for $M/Y$.  Thus, if $X$ is a set of clones of $M$, then
$X$ is a set of clones in any minor $M'$ of $M$ for which
$X\subseteq E(M')$.  The characterization of clones in terms of cyclic
flats makes it evident that the relation on $E(M)$ in which
$e,f\in E(M)$ are related when $e=f$ or $e$ and $f$ are clones is an
equivalence relation; the equivalence classes are the \emph{clonal
  classes} of $M$.

A matroid $M$ of rank $r$ is a \emph{paving matroid} if each subset
$X$ of $E(M)$ with $|X|<r$ is independent.  Thus, each set of size at
most $r-2$ is a flat, so the only flats of a paving matroid $M$ that
may be dependent are $E(M)$ and the hyperplanes.  A paving matroid
that has no dependent hyperplanes is a uniform matroid: in the
\emph{uniform matroid} $U_{r,n}$ on an $n$-element set $E$, the bases
are all $r$-subsets of $E$.  The cyclic flats of a connected paving
matroid $M$ are $\emptyset$, $E(M)$, and (if any) the dependent
hyperplanes.  A paving matroid is \emph{sparse paving} if all of its
dependent hyperplanes are circuits (that is, circuit-hyperplanes). For
example, the rank-$3$ matroids $\mathcal{W}_3$ and $\mathcal{W}^3$
shown in Figure \ref{fig:whirl} and discussed in the next subsection
are sparse paving matroids; in both, $\{1,2,6\}$, $\{2,3,4\}$, and
$\{4,5,6\}$ are circuit-hyperplanes, while $\{1,3,5\}$ is a
circuit-hyperplane of $\mathcal{W}_3$ but not of $\mathcal{W}^3$.  It
follows from Theorem \ref{thm:cfaxioms} that a connected paving
matroid $M$ on $E$ of rank $r\geq 2$ that is not uniform can be
specified by giving subsets $H_1,H_2,\ldots,H_t$ (the dependent
hyperplanes) of $E$ for which (a) $r\leq |H_i|\leq |E|-2$, for all
$i\in [t]$, and (b) $|H_i\cap H_j|\leq r-2$, for all distinct
$i,j\in [t]$; given such sets, letting $\mathcal{Z}$ be
$\{\emptyset, E, H_1,H_2,\ldots,H_t\}$ and setting $r(\emptyset)=0$,
$r(H_i)=r-1$ for all $i\in[t]$, and $r(E)=r$ defines a paving matroid
(condition (a) gives property (Z2) and condition (b) gives the only
nontrivial cases of property (Z3)).

\subsection{Lattice path and multi-path matroids}\label{ssec:lpmmpm}

Two special classes of positroids are the class of lattice path
matroids \cite{lpm} and the larger class of multi-path matroids
\cite{multipath}.  A matroid $M$ is a \emph{lattice path matroid} if
and only if it is transversal and there is a linear order on $E(M)$
and a presentation $\mathcal{A}$ of $M$ that is an antichain of
intervals in $E(M)$, that is, each set in $\mathcal{A}$ is an interval
in $E(M)$ and no set in $\mathcal{A}$ is a subset of another set in
$\mathcal{A}$.  A \emph{multi-path matroid} is a transversal matroid
$M$ that has a presentation by an antichain of cyclic intervals in
some linear order on $E(M)$.

An important matroid that is a multi-path matroid but not a lattice
path matroid is the \emph{$n$-whirl}, $\mathcal{W}^n$, for $n\geq 3$.
This rank-$n$ matroid is formed from the cycle matroid $\mathcal{W}_n$
of an $n$-spoke wheel (an $n$-vertex cycle with one more vertex, the
hub, that is adjacent to all other vertices) by relaxing a
circuit-hyperplane.  (Only when $n=3$ does $\mathcal{W}_n$ have more
than one circuit-hyperplane.)  The sets
$\{1,2,3\}, \{3,4,5\},\ldots,\{2n-1,2n,1\}$ give a presentation of
$\mathcal{W}^n$ by cyclic intervals in the ground set $[2n]$ under the
usual linear order.  See Figure \ref{fig:whirl}.

\begin{figure}
  \centering
  \begin{tikzpicture}[scale=1.2]
   \tikzset{VertexV/.style = {circle, scale=0.8, fill=black!30,
        draw=black, thick}}

    \node[VertexV] (1) at (0,0) {};%
    \node[VertexV] (2) at (0,1) {};%
    \node[VertexV] (3) at (0.866,-0.5) {};%
    \node[VertexV] (4) at (-0.866,-0.5) {};%

    \foreach \from/\to in {1/2,2/3,3/1,1/4,2/4,3/4}%
    \draw[very thick](\from)--(\to);
   
    \node at (0.12,0.3) {\small$6$};%
    \node at (-0.3,-0.33) {\small$2$};%
    \node at (0.3,-0.33) {\small$4$};%
    \node at (0.7,0.1) {\small$5$};%
    \node at (-0.7,0.1) {\small$1$};%
    \node at (0,-0.7) {\small$3$};%

    \draw[thick] (3,1)--(3.866,-0.5)--(2.134,-0.5)--(3,1);%
    \draw [thick] (2.567,0.25) to [out=270,in=180] (3,-0.5) to
    [out=0,in=270] (3.433,0.25);%
    
    \filldraw (3,1) node[above=2] {\small$6$} circle (2.5pt); %
    \filldraw (3,-0.5) node[below=3] {\small$3$} circle (2.5pt);%
    \filldraw (2.567,0.25) node[left=2] {\small$1$} circle (2.5pt);%
    \filldraw (3.866,-0.5) node[below=3] {\small$4$} circle (2.5pt);%
    \filldraw (2.134,-0.5) node[below=3] {\small$2$} circle (2.5pt);%
    \filldraw (3.433,0.25) node[right=2] {\small$5$} circle (2.5pt);%

    \node at (3,-1.2) {\small$\mathcal{W}_3$};%

    \draw[thick] (6,1)--(6.866,-0.5)--(5.134,-0.5)--(6,1);%
    
    \filldraw (6,1) node[above=2] {\small$6$} circle (2.5pt); %
    \filldraw (6,-0.5) node[below=3] {\small$3$} circle (2.5pt);%
    \filldraw (5.567,0.25) node[left=2] {\small$1$} circle (2.5pt);%
    \filldraw (6.866,-0.5) node[below=3] {\small$4$} circle (2.5pt);%
    \filldraw (5.134,-0.5) node[below=3] {\small$2$} circle (2.5pt);%
    \filldraw (6.433,0.25) node[right=2] {\small$5$} circle (2.5pt);%

    \node at (6,-1.2) {\small$\mathcal{W}^3$};%
  \end{tikzpicture}
  \caption{The $3$-spoke wheel, its cycle matroid $\mathcal{W}_3$, and
    the $3$-whirl $\mathcal{W}^3$. The $3$-spoke wheel is the complete
    graph $K_4$, so $\mathcal{W}_3$ is $M(K_4)$.}\label{fig:whirl}
\end{figure}

By \cite[Theorem 6.3]{multipath}, restricting a multi-path matroid to
a proper flat gives a lattice path matroid.  The paragraph after that
result notes that, with that result, some properties of lattice path
matroids automatically extend to multi-path matroids.  For this paper,
the most relevant such property is \cite[Theorem 3.11]{lpm2}: any
connected flat in a lattice path matroid $M$ is an interval in the
linear order on $E(M)$.  Thus, any connected flat in a multi-path
matroid $M$ is a cyclic interval in the linear order on $E(M)$.

\subsection{Amalgams and parallel connections}
For matroids $M$ and $N$, any matroid $K$ on $E(M)\cup E(N)$ for which
$K|E(M)=M$ and $K|E(N)=N$ is an \emph{amalgam} of $M$ and $N$.  The
equality $M|T=N|T$, where $T$ is $E(M)\cap E(N)$, is a necessary (but
not sufficient) condition for amalgams to exist.  As Figure
\ref{fig:amlgamexample} shows, when $M$ and $N$ have an amalgam, there
may be many amalgams; in that example, a different rank-$4$ amalgam
results if $\{d,f,g,i\}$ is a circuit-hyperplane, and another rank-$3$
amalgam has, instead of the line $\{a,e,f,i,h\}$, the lines
$\{a,e,f\}$ and $\{a,i,h\}$.  If $K$ is the freest amalgam of $M$ and
$N$ (i.e., any set that is independent in any amalgam of $M$ and $N$
is independent in $K$), then $K$ is the \emph{free amalgam} of $M$ and
$N$.  Amalgams of $M$ and $N$ may exist without the free amalgam of
$M$ and $N$ existing.  The next result is one implication in
\cite[Proposition 11.4.3]{oxley}.

\begin{thm}\label{thm:free}
  Let $K$ be an amalgam of $M$ and $N$, and set $T=E(M)\cap E(N)$.
  If, for each flat $F$ of $K$, we have
  $$r(F)=r(F\cap E(M))+r(F\cap E(N))-r(F\cap T),$$ then $K$ is the
  free amalgam of $M$ and $N$.
\end{thm}

Note that \cite[Proposition 11.4.3]{oxley} is about the proper
amalgam, for which we refer to \cite[Section 11.4]{oxley} for the
definition and a complete account.  The proper amalgam, when it
exists, is the free amalgam, but the free amalgam can exist without
the proper amalgam existing.  Two general conditions under which the
proper amalgam is known to exist are (a) when $T$ is a modular flat of
either $M$ or $N$, and (b) when the common restriction, $M|T=N|T$, is
a modular matroid.  Condition (b) applies to the amalgams that we
consider here since, in what we treat, $T$ will be independent in both
$M$ and $N$.

When $T$ is the singleton $\{p\}$ and $r_M(p)=r_N(p)$, the free
amalgam of $M$ and $N$ is the \emph{parallel connection}, $P(M,N)$, of
$M$ and $N$ at the base point $p$.  (When $r_M(p)=r_N(p)$, parallel
connection generalizes the operation of gluing two graphs together
along an edge.  This operation is often defined by giving the set of
circuits, as in \cite[Proposition 7.1.4]{oxley}. By \cite[Exercise
7.1.1]{oxley} and Theorem \ref{thm:free}, when $r_M(p)=r_N(p)$, the
parallel connection, as defined in \cite{oxley}, is the free amalgam,
and thinking about parallel connection as the free amalgam is more
useful for this paper, so we take that as the definition.)  If $p$ is
a loop of just one of $M$ and $N$, then although there are no amalgams
of $M$ and $N$, the parallel connection is still defined: if $p$ is a
loop of $M$ but not of $N$, then $P(M,N)$ is defined to be
$M\oplus (N/p)$; if $p$ is a loop of just $N$, then $P(M,N)$ is
defined to be $(M/p)\oplus N$.

\begin{figure}
  \centering
  \begin{tikzpicture}[scale=0.45]
    \draw[thick, black!30](14,-1)--(18,0)--(18,6)--(14,5)--(10,6)
    --(10,0)--(14,-1);%
    \draw[thick, black!30](14,-1)--(14,5);%

    \draw[very thick](0,2)--(3,0)--(3,4)--(0,2);%
    \filldraw (3,2) node[left=1] {\small$b$} circle (5pt);%
    \filldraw (3,0) node[below=1] {\small$a$} circle (5pt);%
    \filldraw (1.5,1) node[below=1] {\small$f$} circle (5pt);%
    \filldraw (0,2) node[above=1] {\small$e$} circle (5pt);%
    \filldraw (1.5,3) node[above=1] {\small$d$} circle (5pt);%
    \filldraw (3,4) node[above=1] {\small$c$} circle (5pt);%

    \draw[very thick](5,0)--(5,4)--(8,2)--(5,0);%
    \filldraw (5,2) node[right =1] {\small$b$} circle (5pt);%
    \filldraw (5,0) node[below =1] {\small$a$} circle (5pt);%
    \filldraw (6.5,1) node[below=1] {\small$i$} circle (5pt);%
    \filldraw (8,2) node[above =1] {\small$h$} circle (5pt);%
    \filldraw (6.5,3) node[above =1] {\small$g$} circle (5pt);%
    \filldraw (5,4) node[above =1] {\small$c$} circle (5pt);%

    \draw[very thick](14,0)--(14,4);%
    \draw[very thick](14,0)--(17,2)--(14,4);%
    \draw[very thick](14,0)--(11,2)--(14,4);%
    
    \filldraw (14,2) node[right =1] {\small$b$} circle (5pt);%
    \filldraw (14,0) node[right =1] {\small$a$} circle (5pt);%
    \filldraw (14,4) node[right =1] {\small$c$} circle (5pt);%
    \filldraw (12.5,1) node[below left] {\small$f$} circle (5pt);%
    \filldraw (12.5,3) node[above left] {\small$d$} circle (5pt);%
    \filldraw (11,2) node[left =1] {\small$e$} circle (5pt);%
    \filldraw (15.5,1) node[below right] {\small$i$} circle (5pt);%
    \filldraw (17,2) node[right =1] {\small$h$} circle (5pt);%
    \filldraw (15.5,3) node[above right] {\small$g$} circle (5pt);%

    \draw[very thick](20,0)--(26,0) -- (23,4)--(20,0);%
    \draw[very thick](23,0)-- (23,4);%

    \filldraw (20,0) node[below =1] {\small$e$} circle (5pt);%
    \filldraw (21.5,0) node[below =1] {\small$f$} circle (5pt);%
    \filldraw (23,0) node[below =1] {\small$a$} circle (5pt);%
    \filldraw (24.5,0) node[below =1] {\small$i$} circle (5pt);%
    \filldraw (26,0) node[below =1] {\small$h$} circle (5pt);%
    \filldraw (23,4) node[above =1] {\small$c$} circle (5pt);%
    \filldraw (23,1.6) node[left =1] {\small$b$} circle (5pt);%
    \filldraw (24.5,2) node[right =1] {\small$g$} circle (5pt);%
    \filldraw (21.5,2) node[left =1] {\small$d$} circle (5pt);%

  \end{tikzpicture}
  \caption{Two rank-$3$ whirls, their free amalgam (which
    has rank $4$), and a rank-$3$ amalgam.  }
  \label{fig:amlgamexample}
\end{figure}

\subsection{Principal extension, series extension, quotients, and
  truncation}
A matroid $N$ is an \emph{extension} of a matroid $M$ if $M$ is a
restriction of $N$.  Single-element extensions have been studied
extensively, starting with Crapo \cite{henry} (see also \cite[Section
7.2]{oxley}).  We focus on principal extensions.  Given a matroid $M$,
a set $X\subseteq E(M)$, and an element $e\not\in E(M)$, define
$r':2^{E(M)\cup e}\to \mathbb{Z}$ by, for all $Y\subseteq E(M)$,
setting $r'(Y)=r_M(Y)$ and
$$r'(Y\cup e)=\begin{cases}
  r_M(Y), & \text{if } X\subseteq\cl_M(Y),\\
  r_M(Y)+1, & \text{otherwise.}
\end{cases}$$ It is routine to show that $r'$ is the rank function of
a matroid on $E(M)\cup e$.  This matroid, which is denoted $M+_X e$,
is the \emph{principal extension} of $M$ in which $e$ is \emph{added
  freely} to $X$.  The \emph{free extension} of $M$ by $e$ is
$M+_{E(M)}e$; we shorten $M+_{E(M)}e$ to $M+e$.  If $r_M(f)=1$, then
$M+_fe$ is a \emph{parallel extension} of $M$.  The parallel extension
$M+_fe$ is the parallel connection of $M$ and the rank-$1$ uniform
matroid on $\{e,f\}$.

Note that $e$ is in the closure of $Y$ in $M+_Xe$ if and only if
either $e\in Y$ or $X\subseteq \cl_M(Y)$.  Also,
$M+_X e = M+_{\cl(X)} e$.  It is easy to check that if
$e,f\not\in E(M)$ and $X$ and $Y$ are subsets of $E(M)$, then
$(M+_Xe)+_Yf=(M+_Yf)+_Xe$, so we do not need to specify an order for
multiple principal extensions.  If $M$ has no loops, then the flat
$\cl_M(X)\cup e$ of $M+_Xe$ is connected by Lemma \ref{lem:conncond};
in particular, $M+e$ is connected.  If $M$ has no loops and
$f\in E(M)$, then there is a minimum nonsingleton connected flat of
$M$ that contains $f$ if and only if $M$ is a principal extension
$(M\del f)+_Xf$ of $M\del f$ for some $X\subseteq E(M\del f)$.

\emph{Series extension} (or, more properly, series coextension) is
dual to parallel extension: $M\times_fe =((M^*)+_fe)^*$.  Thus,
$M\times_fe$ is defined when $f$ is not a coloop of $M$; also,
$r(M\times_fe)=r(M)+1$.  The cocircuits of the parallel extension
$M+_fe$ are those of $M$, except that those that contain $f$ are
augmented by $e$, so in the series extension $M\times_fe$ , the
circuits are those of $M$ except that those that contain $f$ are
augmented by $e$.

For a matroid $M$, a matroid $Q$ with $E(Q)=E(M)$ is a \emph{quotient}
of $M$ if, for some extension $N$ of $M$ with $r(N)=r(M)$, we have
$Q=N/A$ where $A=E(N)-E(M)$. For example, if $M$ is the rank-$4$
matroid in Figure \ref{fig:amlgamexample} and $X=\{a,e,f,i,h\}$, the
quotient $(M+_X j)/j$ is the right-most matroid shown.  If
$A=\{e_1,e_2,\ldots,e_k\}$ with $k\leq r(M)$ and
$N= (((M+e_1)+e_2)\cdots)+e_k$, then $Q=N/A$ is the \emph{truncation}
of $M$ to rank $r(M)-k$, and $r_Q(X)=\min\{r_M(X),r(M)-k\}$ for all
$X\subseteq E(Q)$.  Observe that each basis of a rank-$r$ matroid $M$
is a spanning circuit of its truncation to rank $r-1$, so, if $M$ has
no loops, its truncations to positive ranks less than $r$ are
connected by Lemma \ref{lem:conncond}.

\section{A characterization of positroids and positroid
  orders}\label{sec:char}

Theorem \ref{thm:char} below, the proof of which uses Theorem
\ref{thm:postoh}, characterizes positroids with no loops by
characterizing positroid orders.  The equivalence of statements (1)
and (2) is known (Theorem \ref{thm:bs=p}, due to Lam and Postnikov),
but we include (2) since it fits naturally into our proof of the
equivalence of (1) and (3).  Statement (4), which uses the following
definition, recasts (3) in a way that will prove to be useful.

\begin{dfn}
  For a matroid $M$ that has no loops, a linear order on $E(M)$ has
  the \emph{cyclic interval property} if, for each proper connected
  flat $F$ of $M$ with $|F|\geq 2$ and each connected component $K$ of
  the contraction $M/F$ with $|K|\geq 2$, the set $K$ is a subset of a
  cyclic interval that is disjoint from $F$.
\end{dfn}

We require $|F|\geq 2$ and $|K|\geq 2$ so that we can focus on the
substantial cases when applying Theorem \ref{thm:char}; the condition
trivially holds if either inequality fails.  We assume that $M$ has no
loops since otherwise the condition on connected flats would hold
vacuously.  A matroid is a positroid if and only if the matroid
obtained by deleting all of its loops is a positroid, and the
placement of loops has no impact on whether a linear order is a
positroid order, so the assumption of no loops does not limit the
applicability of Theorem \ref{thm:char}.  In this result, the matroid
is not assumed to be connected.

\begin{thm}\label{thm:char}
  Let $M$ be a matroid that has no loops and let $\leq$ be a linear
  order on $E(M)$.  The following statements are equivalent:
  \begin{itemize}
  \item[\emph{(1)}] $\leq$ is a positroid order for $M$,
  \item[\emph{(2)}] $\leq$ is a base-sorting order for $M$,
  \item[\emph{(3)}] if $F$ is a flat of $M$ for which
    $2\leq |F|\leq |E(M)|-2$ and $M|F$ is connected, then
    \begin{equation}\label{eq:dirsum}
      M/F = (M/F|U_1)\oplus (M/F|U_2)\oplus \cdots \oplus (M/F|U_h)
    \end{equation}
    where $U_1,U_2,\ldots,U_h$ are the maximal cyclic intervals that
    are disjoint from $F$, and
  \item[\emph{(4)}] $\leq$ has the cyclic interval property.
 \end{itemize}
\end{thm}

\begin{proof}
  It is easy to see that statements (3) and (4) are equivalent.  Below
  we show that statements (1)--(3) are equivalent.  Let $n=|E(M)|$ and
  $r=r(M)$.

  Assume that statement (1) holds.  By Theorem \ref{thm:postoh}, the
  set of bases of $M$ is
  $$\bigcap_{i\in[n]}\mathcal{B}\bigl(N(I_i,\leq_i)\bigr)$$ where
  $I_i$ is the Gale basis of $M$ in the $i$-shift $\leq_i$ of $\leq$.
  Let $B$ and $B'$ be bases of $M$, and let
  $\sort_\leq(B,B')=\{B_e,B_o\}$.  To show that $B_e$ and $B_o$ are
  bases of $M$, we show that
  $B_e,B_o\in\mathcal{B}\bigl(N(I_i,\leq_i))$ for all $i\in[n]$.  Now
  $B,B'\in\mathcal{B}\bigl(N(I_i,\leq_i))$ since both $B$ and $B'$ are
  bases of $M$, so, for each $j\in [r]$, each of $B$ and $B'$ has at
  most $j-1$ elements that are less than the $j$th element of $I_i$
  using $\leq_i$; thus, the elements in positions $2j-1$ and $2j$ when
  we list the elements in the multiset union $B\cup B'$ in order,
  using $\leq_i$, are both at least the $j$th element of $I_i$.  The
  elements in positions $2j-1$ and $2j$ are the $j$th elements of
  $B_o$ and $B_e$, so $B_e,B_o\in \mathcal{B}\bigl(N(I_i,\leq_i))$, so
  statement (2) holds.

  Now assume that statement (2) holds.  Let $F$ be a connected flat
  $F$ of $M$ for which $2\leq |F|\leq |E(M)|-2$.  Let
  $U_1,U_2,\ldots,U_h$ be the maximal cyclic intervals that are
  disjoint from $F$.  By replacing $\leq$ by a shift $\leq_j$ if
  needed, we may assume that the elements of $E(M)$ are, in order,
  $$f^1_1,f^1_2,\ldots,f^1_{p_1},
  u^1_1,u^1_2,\ldots,u^1_{q_1},\ldots, f^h_1,f^h_2,\ldots,f^h_{p_h},
  u^h_1,u^h_2,\ldots,u^h_{q_h},$$ where all $p_i,q_i\in \mathbb{N}$
  and $f^i_j\in F$, and $U_i=\{ u^i_1,u^i_2,\ldots,u^i_{q_i}\}$.  Set
  $F_i = \{f^i_1,f^i_2,\ldots,f^i_{p_i}\}$ for $i\in [h]$.  Statement
  (3) holds trivially if $h=1$, so assume that $h>1$.

  Since $M|F$ is connected, there is a circuit $C\subseteq F$ with
  $\{f^1_{p_1},f^2_1\}\subseteq C$.  Let $B$ be a basis of $M|F$ with
  $C-f^1_{p_1}\subseteq B$.  Let $B_1$ and $B_2$ be bases of $M$ with
  $B\subseteq B_1\cap B_2$.  Let $t$ be
  $|B_1\cap U_1|+ |B_2\cap U_1|$.  We claim that $t$ is even, so
  $|B_1\cap U_1|$ and $|B_2\cap U_1|$ have the same parity.  To see
  this, let the elements in the multiset union $B_1\cup B_2$, listed
  in order, be
  $$a_1=a_1<a_2=a_2<\cdots <a_i=a_i< z_1\leq z_2\leq \cdots\leq z_t<
  f^2_1=f^2_1< \cdots
  $$
  where $\{a_1,a_2,\ldots,a_i\}\subseteq F_1-f^1_{p_1}$ and
  $z_j\in U_1$ for $j\in[t]$.  Since
  $\{f^1_{p_1},f^2_1\}\subseteq C$ and $C-f^1_{p_1}\subseteq B$, the
  set $B'_1=(B_1-f^2_1)\cup f^1_{p_1}$ is a basis of $M$.  The
  elements in the multiset union $B'_1\cup B_2$, listed in order, are
  $$a_1= a_1<a_2=a_2<\cdots <a_i= a_i<f^1_{p_1}< z_1\leq
  z_2\leq\cdots\leq z_t< f^2_1< \cdots .
  $$
  Both lists have $2 r(F)$ elements of $F$, so since the sets in
  $\sort_\leq(B_1, B_2)$ and $\sort_\leq(B'_1, B_2)$ are bases, in
  each list, $r(F)$ elements of $F$ are in even positions and the
  other $r(F)$ elements of $F$ are in odd positions.  The only
  difference in the positions of elements in $F$ in the two lists is
  that the first copy of $f^2_1$ in the first list changes to
  $f^1_{p_1}$ and moves $t$ places earlier.  Therefore $t$ must be
  even, and so $|B_1\cap U_1|$ and $|B_2\cap U_1|$ have the same
  parity.
  
  We next draw a sharper conclusion: $|B_1\cap U_1| = |B_2\cap U_1|$
  for any bases $B_1$ and $B_2$ of $M$ with $B\subseteq B_1\cap B_2$.
  This holds since we can get $B_2$ from $B_1$ by a sequence of
  single-element exchanges, and no such exchange involves elements of
  $F$ (since $B\subseteq B_1\cap B_2$), so, by the parity result
  above, each exchange either exchanges an element in $U_1$ for an
  element in $U_1$, or an element in $E(M)-(F\cup U_1)$ for an element
  in $E(M)-(F\cup U_1)$.

  For any two bases $B$ and $B'$ of $M|F$, and any subset $I$ of
  $E(M)-F$, the set $B\cup I$ is a basis of $M$ if and only if
  $B'\cup I$ is a basis of $M$.  This gives the stronger conclusion
  that for all bases $B_1$ and $B_2$ of $M$ that contain bases of
  $M|F$, we have $|B_1\cap U_1| = |B_2\cap U_1|$.
  
  The same argument applied to cyclic shifts shows that the same
  result holds for each of $U_2,\ldots,U_h$.  For each $j\in[h]$, let
  $s_j=r(F\cup U_j) - r(F)$, so $s_j=r(M/F|U_j)$.  We can extend a
  basis $B$ of $M|F$ to a basis of $M|(F\cup U_j)$, and then extend
  that to a basis of $M$, so what we just showed implies that for each
  basis $B_1$ of $M$ with $B\subseteq B_1$, we have
  $|B_1\cap U_j|=s_j$ for each $j\in[h]$.  Thus,
  $r(M) =r(F)+s_1+s_2+\cdots +s_h$, and so
  \begin{equation}\label{eq:rank}
    r(M/F) = r(M/F|U_1)+ r(M/F|U_2)+ \cdots + r(M/F|U_h),
  \end{equation}
  and so the direct sum decomposition in statement (3) follows by
  Lemma \ref{lem:dsdecomprank}.
  
  Now assume that statement (3) holds.  Recall that $n=|E(M)|$ and
  $r=r(M)$.  Let $I_i$ be the Gale basis of $M$ using $\leq_i$.  By
  Theorem \ref{thm:postoh}, the matroid $M'$ for which
  $$\mathcal{B}(M') =
  \bigcap_{i\in[n]}\mathcal{B}\bigl(N(I_i,\leq_i)\bigr)$$ is a
  positroid.  Since $I_i$ is the Gale basis of $M$ using $\leq_i$, we
  have $\mathcal{B}(M)\subseteq \mathcal{B}\bigl(N(I_i,\leq_i)\bigr)$
  by Lemma \ref{lem:GaleIsMin}, so
  $\mathcal{B}(M)\subseteq\mathcal{B}(M')$.  Showing equality in that
  inclusion will give $M=M'$, and so $M$ is a positroid and $\leq$ is
  a positroid order, thus proving statement (1).  To show that
  equality, it suffices to show that no $r$-subset of $E(M)$ that
  contains a (necessarily non-spanning) circuit of $M$ is in
  $\mathcal{B}(M')$.  Let $C$ be a non-spanning circuit of $M$, let
  $F$ be the connected flat $\cl_M(C)$ of $M$, and let
  $C\subseteq D\subseteq E(M)$ where $|D|=r$.  Let
  $U_1,U_2,\ldots,U_h$ be the maximal cyclic intervals that are
  disjoint from $F$, and let $s_j$ be $r(M/F|U_j)$.  Equation
  (\ref{eq:dirsum}) gives equation (\ref{eq:rank}), so
  $r(M) =r_M(F)+s_1+s_2+\cdots + s_h$.  Since
  $|D\cap F|\geq |C|>r_M(F)$, we have $|D\cap U_j|<s_j$ for at least
  one set $U_j$.  Pick the $i$ for which $U_j$ is a filter in the
  $i$-order $\leq_i$.  Equation (\ref{eq:dirsum}) gives
  $M/F|U_j =M/(E(M)-U_j)$, so the Gale basis $I_i$ must contain $s_j$
  elements of $U_j$. Thus, each basis of $N(I_i,\leq_i)$ contains at
  least $s_j$ elements in the filter $U_j$ of the $i$-order $\leq_i$.
  Since $|D\cap U_j|<s_j$, we have
  $D\not\in\mathcal{B}(N(I_i,\leq_i))$, so $D\not\in\mathcal{B}(M')$,
  as needed.
\end{proof}

The corollary below is obtained from condition (3) in Theorem
\ref{thm:char} by duality and the following results: $F$ is a flat of
$M$ if and only if $E(M)-F$ is a cyclic set of $M^*$; the positroid
orders for $M$ are the same as the positroid orders for $M^*$.

\begin{cor}\label{cor:chardual}
  Let $M$ be a matroid that has no coloops.  A linear order $\leq$ on
  $E(M)$ is a positroid order for $M$ if and only if, for all cyclic
  sets $A$ of $M$ for which $M/A$ is connected and
  $2\leq |A|\leq |E(M)|-2$, we have
  $M|A = (M|V_1)\oplus (M|V_2)\oplus \cdots \oplus (M|V_h)$ where
  $V_1,V_2,\ldots,V_h$ are the maximal cyclic intervals that are
  subsets of $A$.
\end{cor}

The next corollary is useful both for limiting the search for
potential positroid orders and for deducing that certain matroids are
not positroids.  This result is part of \cite[Proposition 5.6]{ARW} by
Ardila, Rinc\'on, and Williams, as cast in Theorem \ref{thm:arw2}
above.

\begin{cor}\label{cor:convinterval}
  Let $F$ be a flat in a matroid $M$ that has no loops.  If both $M|F$
  and $M/F$ are connected, then $F$ is a cyclic interval in any
  positroid order for $M$.
\end{cor}

\begin{proof}
  This follows since the integer $h$ in condition (3) in Theorem
  \ref{thm:char} must be $1$.
\end{proof}

For example, the cycle matroid $M(K_4)$ is not a positroid since the
four $3$-point lines are connected, as are the corresponding
contractions, but, by Lemma \ref{lem:noco}, there is no linear order
on the ground set in which each of those lines is a cyclic interval.
(The example before Lemma \ref{lem:noco} uses the $3$-point lines of
$M(K_4)$ as labeled in Figure \ref{fig:whirl}.)  This idea is taken
much further in Examples \ref{example:genK4}--\ref{ex:K4last} of
Section \ref{sec:exmin}.

\begin{example}\label{rank5example}
  Let $M$ be the cycle matroid of the graph in Figure
  \ref{fig:cyclicvscyclic}, which has rank five and is the parallel
  connection of four copies of the uniform matroid $U_{2,3}$, one on
  each of the sets $\{3,6,9\}$, $\{1,2,3\}$, $\{4,5,6\}$, and
  $\{7,8,9\}$.  By Theorem \ref{thm:char}, $M$ is a positroid since
  the cyclic interval property clearly holds for each proper connected
  flat with at least two elements that is a cyclic interval (so,
  $\{1,2,3\}$, $\{4,5,6\}$, and $\{7,8,9\}$), and all other connected
  flats $F$ with $|F|\geq 2$ contain $\{3,6,9\}$, so each connected
  component of $M/F$ is one of the intervals $\{1,2\}$, $\{4,5\}$, and
  $\{7,8\}$.

  To illustrate a point that is used at the end of the proof of
  Theorem \ref{thm:char}, we explain how each $5$-element subset $X$
  of $E(M)$ that contains a non-spanning circuit of $M$ fails to be a
  basis of some matroid $N(I_i,\leq_i)$.  If $\{1,2,3\}\subseteq X$,
  then $X$ is not a basis of $N(I_1,\leq_1)$ since $I_1=\{1,2,4,5,7\}$
  but at most two elements $j\in X$ satisfy $4\leq_1 j$.  The cases of
  $\{4,5,6\}\subseteq X$ (use $I_4$) and $\{7,8,9\}\subseteq X$ (use
  $I_7$) are similar.  If $X$ contains a $4$- or $5$-circuit, then $X$
  is disjoint from $\{1,2\}$, $\{4,5\}$, or $\{7,8\}$.  If
  $X\cap \{1,2\}=\emptyset$, then $X$ is not a basis of
  $N(I_3,\leq_3)$ since the fifth element of $X$ is less than that of
  $I_3=\{3,4,5,7,1\}$ in $\leq_3$.  The cases of
  $X\cap \{4,5\}=\emptyset$ (use $I_6$) and $X\cap \{7,8\}=\emptyset$
  (use $I_9$) are similar.

  Figure \ref{fig:counterexample} shows the rank-$4$ truncation of
  $M$.  Applying Corollary \ref{cor:convinterval} to the connected
  flat $\{3,6,9\}$ shows that the usual order is not a positroid order
  for the truncation of $M$ to rank $3$ or $4$.  Indeed, these
  truncations are excluded minors for the class of positroids (see
  Section \ref{sec:exmin}).  Thus, the class of positroids is not
  closed under truncation. \hfill$\circ$
\end{example}

\begin{figure}
  \centering
  \begin{tikzpicture}[scale=1]
    \draw[thick,gray!60](0,0)--(0,2)--(1.5,3)--(1.5,1)--(0,0)
    --(-1.5,1)--(-1.5,3)--(0,2);%
    \draw[very thick](0,0.3)--(1.3,1.2);%
    \draw[very thick](0,0.3)--(0,1.7);%
    \draw[very thick](0.5,3.3)--(0,1.7);%
    \draw[very thick](0,1)--(-1.3,1.8);%
    \filldraw (0,0.3) node[above left =-1] {\small$3$} circle
    (2.75pt);%
    \filldraw (0,1) node[above right=-1] {\small$6$} circle (2.75pt);%
    \filldraw (1.3,1.2) node[above=1] {\small$1$} circle (2.75pt);%
    \filldraw (0.65,0.75) node[above=1] {\small$2$} circle (2.75pt);%
    \filldraw (0,1.7) node[right=1] {\small$9$} circle (2.75pt);%
    \filldraw (-1.3,1.8) node[above =1] {\small$4$} circle (2.75pt);%
    \filldraw (-0.65,1.4) node[above =1] {\small$5$} circle (2.75pt);%
    \filldraw (0.5,3.3) node[left =1] {\small$7$} circle (2.75pt);%
    \filldraw (0.25,2.5) node[left =1] {\small$8$} circle (2.75pt);%
  \end{tikzpicture}
  \caption{A rank-$4$ excluded minor for the class of positroids that
    is the truncation of a rank-$5$ positroid.}
  \label{fig:counterexample}
\end{figure}

The proof of the next result shows that (1) the cyclic interval
property, (2) a condition that Blum conjectured to characterize
base-sorting orders \cite[Conjecture 4.9]{blum}, and (3) a
characterization of positroid orders via forbidden induced orders on
certain $4$-element rank-$2$ minors can easily be derived from each
other.  The third characterization, while widely known among
researchers in the field, seems to never have appeared in the
literature.

\begin{thm}\label{thm:equivchars}
  Let $M$ be a matroid that has no loops and let $\leq$ be a linear
  order on $E(M)$.  The following statements are equivalent:
  \begin{itemize}
  \item[\emph{(1)}] $\leq$ has the cyclic interval property;
  \item[\emph{(2)}] for every minor $N$ of $M$ that has neither loops
    nor coloops, each circuit-hyperplane of $N$ is a cyclic interval
    in the linear order that $\leq$ induces on $E(N)$;
  \item[\emph{(3)}] for every minor $N$ of $M$ for which $E(N)$ is the
    disjoint union of a $2$-element circuit $C$ and a $2$-element
    cocircuit, the circuit $C$ is a cyclic interval in the linear
    order that $\leq$ induces on $E(N)$.
  \end{itemize}
\end{thm}

\begin{proof}
  Assume that statement (1) holds, so $\leq$ is a positroid order by
  Theorem \ref{thm:char}. Let $N$ be a minor of $M$ that has neither
  loops nor coloops and let $H$ be a circuit-hyperplane of $N$.  By
  Corollary \ref{cor:convinterval} applied to the positroid $N$ and
  its connected flat $H$, the set $H$ is a cyclic interval in the
  linear order that $\leq$ induces on $E(N)$.  Thus, statement (2)
  holds, as does (3), which is a special case of (2).

  Now assume that statement (1) fails.  We claim that statement (3),
  and hence (2), also fail.  Let $F$ be a connected flat of $M$ and
  let $K$ be a connected component of $M/F$ that is not contained in
  any cyclic interval of $E(M)$ that is disjoint from $F$.  Thus,
  there are elements $a,b\in F$ and $e,f\in K$ for which $\{a,b\}$ is
  not a cyclic interval in the induced order on $\{a,b,e,f\}$.  Since
  $F$ is connected, some circuit $C$ of $M|F$ has $a,b\in C$.  Since
  $M/F|K$ is connected, some cocircuit $C^*$ of $M/F$ has
  $e,f\in C^*$, and $C^*$ is a cocircuit of $M$ by Lemma
  \ref{lem:cocofcon}.  Extend $C-a$ to a basis $B$ of $M-C^*$.  Set
  $N=M/(B- b)| \{a,b,e,f\}$.  Thus, $N$ has rank two, $\{a,b\}$ is a
  circuit of $N$, and $\{e,f\}$ is a cocircuit of $N$.  Since
  $\{a,b\}$ is not a cyclic interval in the induced order on
  $\{a,b,e,f\}$, statement (3) fails, as claimed.
\end{proof}

Statement (3) in the result above is an excluded-minor
characterization of positroids if one considers the linear order to be
part of the positroid.  In contrast, in this paper, we distinguish
between the matroid and positroid orders for it, so the excluded
minors that we consider in Section \ref{sec:exmin} are for the class
of matroids for which positroid orders exist.

The cyclic interval property clearly holds under the hypotheses of the
next corollary, which are less restrictive than the hypotheses of
\cite[Proposition 4.5]{blum}.

\begin{cor}\label{cor:intervalsimpliespositroid}
  Let $M$ be a matroid with no loops.  If there is a linear order on
  $E(M)$ in which each connected flat of $M$ is a cyclic interval,
  then $M$ is a positroid and the linear order is a positroid order.
\end{cor}

In particular, any matroid with no loops in which the proper
nonsingleton connected flats are pairwise disjoint is a positroid.
Thus, every rank-$2$ matroid is a positroid.

If a matroid $M$ satisfies the hypothesis of Corollary
\ref{cor:intervalsimpliespositroid}, then so do its truncations since
all proper connected flats of a truncation are proper connected flats
of $M$.  The same does not apply to quotients other than truncations.
Also, the same is not true of Theorem \ref{thm:char}.  Indeed, the
class of positroids is not closed under truncations, as Example
\ref{rank5example} shows.

Oh \cite{oh} proved that lattice path matroids are positroids.  From
Corollary \ref{cor:intervalsimpliespositroid} and the remarks in
Section \ref{ssec:lpmmpm}, noting our observation on loops at the
beginning of this section, it is easy to see that the same is true of
the larger class of multi-path matroids (this also follows from Blum
\cite[Theorem 5.2]{blum}), and likewise for truncations of lattice
path matroids (which need not be lattice path), and truncations of
multi-path matroids (which need not be multi-path).  As an aside, we
note that the class of lattice path matroids is closed under direct
sums, but that of multi-path matroids is not; also, truncations of
direct sums of multi-path matroids (e.g., truncations of direct sums
of whirls) need not be positroids.

We next show that the class of positroids is closed under
circuit-hyperplane relaxation and a more general relaxation operation
that we now discuss.  Assume that a cyclic flat $X$ of a matroid $M$
is neither $\cl_M(\emptyset)$ nor $E(M)-\cl_{M^*}(\emptyset)$, that
the only cyclic flat that properly contains $X$ is
$E(M)-\cl_{M^*}(\emptyset)$, and that the only cyclic flat that $X$
properly contains is $\cl_M(\emptyset)$.  (Thus, $X$ could be a
circuit-hyperplane of a connected matroid of rank at least two, but
$X$ is not limited to such sets.)  Observe that properties (Z0)--(Z3)
in Theorem \ref{thm:cfaxioms} hold for
$\mathcal{Z}=\mathcal{Z}(M)-\{X\}$ and $r:\mathcal{Z}\to\mathbb{Z}$
where $r(A)=r_M(A)$ for each $A\in\mathcal{Z}$.  By that result, this
reduced cyclic flat and rank data defines another matroid, a
generalized relaxation of $M$, in which $X$ is independent.

\begin{cor}
  Assume that $X$ is a proper, nonempty cyclic flat of a matroid $M$
  that has no loops and no coloops, and that no other proper nonempty
  cyclic flat of $M$ either contains $X$ or is contained in $X$.  Let
  $M'$ be the matroid for which
  $\mathcal{Z}(M')=\mathcal{Z}(M) -\{X\}$ and $r_{M'}(A)=r_M(A)$ for
  all $A\in \mathcal{Z}(M')$.  If $M$ is a positroid, then so is $M'$,
  and any positroid order for $M$ is a positroid order for $M'$.
  Thus, any circuit-hyperplane relaxation of a connected positroid of
  rank at least two is a positroid.
\end{cor}

\begin{proof}
  It follows from Lemma \ref{lem:cyclicminor} that for any proper,
  nonempty cyclic flat $F$ of $M'$, we have $M|F=M'|F$ and $M/F=M'/F$.
  Thus, the proper connected flats $F$ of $M'$ with $|F|\geq 2$ (all
  of which are cyclic flats of $M'$) are those of $M$ other than $X$.
  For a positroid order for $M$, the cyclic interval property holds
  for $M$, so the cyclic interval property also holds for $M'$, so the
  linear order is a positroid order for $M'$.
\end{proof}

The same result and proof apply to the more general relaxation $M'$ of
$M$ where $\mathcal{Z}(M')$ is $\mathcal{Z}(M)-\mathcal{S}$ for some
$\mathcal{S}\subseteq \mathcal{Z}(M)-\{\emptyset,E(M)\}$ where
$\mathcal{S}\cup\{\emptyset,E(M)\}$ is both a filter and an ideal of
$\mathcal{Z}(M)$.

Since the class of positroids is closed under contraction but not
under truncation, it is not closed under free extension. (It is
therefore also not closed under the operation of free product, of
which free extension is a special case; see \cite{free1,free2,oxley}.)
The following corollary characterizes when free extensions of
positroids are positroids.

\begin{cor}\label{cor:freeextcor}
  Let $M$ be a matroid with no loops, and fix $e\not\in E(M)$.  The
  free extension $M+e$ is a positroid if and only if there is a linear
  order on $E(M)$ in which each proper connected flat of $M$ is an
  interval.
\end{cor}

\begin{proof}
  First assume that $M+e$ is a positroid, and let $\leq$ be a
  positroid order for $M+e$.  By replacing $\leq$ by a cyclic shift if
  needed, we may assume that $e$ is the greatest element of $E(M+e)$.
  Since $e$ is neither a coloop of $M+e$ nor in any proper connected
  nonsingleton flat of $M+e$, by Lemma \ref{lem:conncond} any
  contraction $(M+e)/F$ by a flat $F$ with $e\not\in F$ is connected.
  Let $F$ be a proper connected flat of $M$.  Then $F$ is a proper
  connected flat of $M+e$, and $(M+e)/F$ is connected, so $F$ is a
  cyclic interval in $\leq$ by Corollary \ref{cor:convinterval}.
  Therefore $F$ is an interval in $\leq$ since $e\not\in F$.  Thus, in
  the positroid order that $\leq$ induces on $E(M)$, every proper
  connected flat of $M$ is an interval.

  The converse follows easily from Corollary
  \ref{cor:intervalsimpliespositroid} since the proper connected
  nonsingleton flats of $M$ are precisely those of $M+e$, so if there
  is a linear order on $E(M)$ in which each proper connected flat of
  $M$ is an interval, then the same holds for $M+e$ by extending that
  positroid order for $M$ to make $e$ the greatest element in the
  linear order.
\end{proof}

\begin{cor}
  If the ground sets of $M_1,M_2,\ldots,M_t$ are pairwise disjoint and
  the free extension $M_i+e$, for each $i\in[t]$, is a positroid, then
  any iterated free extension of any truncation of
  $M_1\oplus M_2\oplus \cdots\oplus M_t$ is a positroid.
\end{cor}

\begin{proof} Let $N=M_1\oplus M_2\oplus \cdots\oplus M_t$.  For each
  $i\in[t]$, since $M_i+e$ is a positroid, there is a linear order
  $e_{i,1},e_{i,2},\ldots,e_{i,n_i}$ on $E(M_i)$ in which each
  connected flat of $M_i$ is an interval.  Concatenating these linear
  orders as
  $$e_{1,1},e_{1,2},\ldots,e_{1,n_1},
  e_{2,1},e_{2,2},\ldots,e_{2,n_2},\ldots,
  e_{t,1},e_{t,2},\ldots,e_{t,n_t}$$ gives a linear order on $E(N)$ in
  which each connected flat of $N$ is an interval.  By Corollary
  \ref{cor:freeextcor}, adding elements freely to $N$ yields a
  positroid, and contracting some of them gives an iterated free
  extension of a truncation of $N$.
\end{proof}
    
We next identify some cases in which principal extensions of
positroids are positroids.  The assumption in this result amounts to
$M$ being a principal extension of $M\del f$.

\begin{cor}
  Let $M$ be a positroid with no loops. Fix $e\not\in E(M)$.  Assume
  that for some nonsingleton connected flat $A$ of $M$ and
  $f\in E(M)$, for each nonsingleton connected flat $F$ of $M$, we
  have $f\in F$ if and only if $A\subseteq F$.  Then the principal
  extension $M+_Ae$ of $M$ is a positroid.
\end{cor}

\begin{proof}
  By construction, $e$ and $f$ are clones in $M+_Ae$.  Let $\leq$ be a
  positroid order for $M$.  We may assume that $f$ is the greatest
  element of $E(M)$ under $\leq$.  Define $\leq'$ on $E(M)\cup e$ by
  appending $e$ as the new greatest element, so $E(M)$ is an ideal in
  the extension $\leq'$ of $\leq$.  The connected flats of $M+_Ae$
  with at least two elements are of two types: (i) connected flats $F$
  of $M$ with $|F|\geq 2$ and $f\not\in F$, and (ii) $F\cup e$ where
  $F$ is a connected flat of $M$ with $|F|\geq 2$ and $f\in F$.  For a
  connected flat $F$ of $M+_Ae$ of type (i), the connected components
  of $(M+_Ae)/F$ are those of $M/F$, but with the one that contains
  $f$ augmented by $e$.  Also, for a connected flat $F\cup e$ of
  $M+_Ae$ of type (ii), the connected components of
  $(M+_Ae)/(F\cup e)$ are those of $M/F$.  With those observations, it
  is easy to see that $\leq'$ has the cyclic interval property.
\end{proof}

There are infinite antichains of positroids in the minor quasi-order
(where $N\leq M$ if $N$ is isomorphic to a minor of $M$), so the class
of positroids is not well-quasi-ordered.  Some examples of infinite
antichains of positroids are: the excluded minors for the class of
nested matroids (the truncation, to rank $n$, of the direct sum of two
$n$-circuits, for $n\geq 2$; see \cite{nested}); apart from $M(K_4)$,
the excluded minors for the class of lattice path matroids (see
\cite{lpmexmin}); for all $n\geq 3$, the truncation of the rank-$n$
whirl $\mathcal{W}^n$ to rank $3$.
 
The next result will be useful in Section \ref{sec:bonding}.

\begin{cor}\label{cor:clone}
  Let $X_1,X_2,\ldots,X_d$ be pairwise disjoint sets of clones in a
  positroid $M$ for which, in some positroid order for $M$, some
  $x\in X_1$ and $y\in X_2$ are cyclically consecutive.  Then there is
  a positroid order for $M$ in which each $X_i$, for $i\in[d]$, is a
  cyclic interval, as is $X_1\cup X_2$.  In particular, there is a
  positroid order for $M$ in which each clonal class is a cyclic
  interval.
\end{cor}

\begin{proof}
  Assume that the following linear order, $\leq$, on $E(M)$, written
  as a list from the minimum to the maximum element, is a positroid
  order for $M$:
  $$x^1_1,\ldots, x^1_{t_1},a^1_1,\ldots,a^1_{s_1}, x^2_1,\ldots ,
  x^2_{t_2},a^2_1,\ldots,a^2_{s_2},\ldots, x^k_1,\ldots ,
  x^k_{t_k},a^k_1,\ldots,a^k_{s_k},$$ where all $t_i$, $s_i$ are in
  $\mathbb{N}$, each $x_i^j$ is in $X_1$, and each $a_i^j$ is in
  $E(M)-X_1$.  Also, we can take the elements $x$ and $y$ in the
  hypothesis to be $x=x_1^1$ and $y=a_{s_k}^k$.  (By using a cyclic
  shift and the order dual as needed, any other pair of cyclically
  consecutive elements $x\in X_1$ and $y\in X_2$ can be brought into
  the analogous position.)  
  Consider the linear order $\leq_1$ that, written as
  a list, is given by
  $$x^1_1,\ldots, x^1_{t_1}, x^2_1,\ldots ,
  x^2_{t_2},\ldots,x^k_1,\ldots ,
  x^k_{t_k},a^1_1,\ldots,a^1_{s_1},a^2_1,\ldots,a^2_{s_2},\ldots,
  a^k_1,\ldots,a^k_{s_k},$$ so $\leq$ and $\leq_1$ induce the same
  linear orders on $X_1$ and on $E(M)-X_1$; also, $X_1$ is an ideal of
  $\leq_1$.  If $k=1$, then $\leq_1$ is the positroid order $\leq$ for
  $M$.  We next show that $\leq_1$ is a positroid order for $M$ even
  if $k>1$.
  
  Let $F$ be a connected flat of $M$ with $|F|\geq 2$ and let $K$ be a
  connected component of $M/F$ with $|K|\geq 2$.  By Theorem
  \ref{thm:char}, there is a cyclic interval $I$ for $\leq$ so that
  $K\subseteq I$ and $I\cap F=\emptyset$.  We must show that
  $K\subseteq I'$ and $I'\cap F=\emptyset$ for some cyclic interval
  $I'$ for $\leq_1$.  If $X_1\subseteq F$, then since $F$ and the
  cyclic interval $I$ are disjoint,
  $I\subseteq \{a_1^i,a_2^i,\ldots,a_{s_i}^i\}$ for some $i\in [k]$,
  and so $I$ is a cyclic interval for $\leq_1$ and we can take $I'=I$.
  Now assume that $X_1\not\subseteq F$. Then $X_1\cap F=\emptyset$
  since $X_1$ is a set of clones of $M$.  Since $X_1$ is also a set of
  clones of $M/F$, either $X_1\cap K=\emptyset$ or $X_1\subseteq K$.
  First assume that $X_1\cap K=\emptyset$.  If
  $\{x^1_1,\ldots, x^1_{t_1}\}\cap I=\emptyset$, then $I'=I-X_1$ is a
  cyclic interval for $\leq_1$ and satisfies $K\subseteq I'$ and
  $I'\cap F=\emptyset$.  If
  $\{x^1_1,\ldots, x^1_{t_1}\}\cap I\ne\emptyset$, then $I'=I\cup X_1$
  is a cyclic interval for $\leq_1$ and satisfies $K\subseteq I'$ and
  $I'\cap F=\emptyset$.  Now assume that $X_1\subseteq K$. Then
  $X_1\subseteq I$, so the cyclic interval $E(M)-I$ for $\leq$, which
  contains $F$, is $\{ a^h_p, a^h_{p+1}, \ldots,a^h_q\}$ for some $h$,
  $p$, and $q$.  Since $\{ a^h_p, a^h_{p+1}, \ldots,a^h_q\}$ is also a
  cyclic interval for $\leq_1$, so is $I$, so we can take $I'=I$.
  
  Starting with the dual of $\leq_1$, apply the argument above to
  $X_2$ in place of $X_1$, keeping $x=x_1^1$ and $y=a_{s_k}^k$
  cyclically consecutive, giving $\leq_2$.  Since the restrictions of
  the orders $\leq_1$ and $\leq_2$ to $E(M)-X_2$ are the same, $X_1$
  is a cyclic interval in $\leq_2$, as are $X_2$ and $X_1\cup X_2$.
  To complete the proof, apply the same modification of the linear
  order, one step at a time, for each of $X_3,\ldots,X_d$.
\end{proof}

We close this section with a corollary of Theorem \ref{thm:char} that
we will use in Section \ref{sec:exmin}.

\begin{cor}\label{cor:3cyc}
  Let $M$ be a matroid that has exactly three proper, nonempty cyclic
  flats, $Z_1$, $Z_2$, and $Z_3$. If $E(M)=Z_1\cup Z_2\cup Z_3$ and
  $Z_1\cap Z_2\cap Z_3=\emptyset$, then $M$ is a positroid.
\end{cor}

\begin{proof}
  Any proper connected flat with at least two elements is cyclic and
  so must be $Z_1$, $Z_2$, or $Z_3$, so, by Corollary
  \ref{cor:intervalsimpliespositroid}, it suffices to find a linear
  order on $E(M)$ in which $Z_1$, $Z_2$, and $Z_3$ are cyclic
  intervals.  Consider the following six subsets of $E(M)$, some of
  which may be empty: $Z_1\cap Z_3$, $Z_1-(Z_2\cup Z_3)$,
  $Z_1\cap Z_2$, $Z_2-(Z_1\cup Z_3)$, $Z_2\cap Z_3$, and
  $Z_3-(Z_1\cup Z_2)$; for each set, pick a linear order on the
  elements, and then concatenate these six linear orders in the order
  given for these sets (so, for example, $Z_1\cap Z_3$ is an ideal in
  this linear order).  The result is a positroid order for $M$.
\end{proof}

Similarly, any matroid that has at most four cyclic flats (and so at
most one pair of incomparable cyclic flats) is a positroid.

\section{Bonding and its application to amalgams of
  positroids}\label{sec:bonding}

In this section, we prove Theorem \ref{thm:bondmain1}: if $M$ and $N$
are positroids and $E(M)\cap E(N)$ is an independent set of clones in
$M$ and in $N$, then the free amalgam of $M$ and $N$ is a positroid.
We also prove a second result of this type (Theorem
\ref{thm:bondmain2}).

A difficulty that arises is that a contraction of an amalgam of $M$
and $N$ need not be an amalgam of the corresponding contractions of
$M$ and $N$.  To deal with this, we introduce a more general way to
glue matroids $M$ and $N$ together along $T=E(M)\cap E(N)$ that
coincides with the free amalgam when $T$ is independent in both $M$
and $N$, but that has crucial contraction properties that come from
the greater generality of the construction.

This section has two parts.  In the first, we define the new
construction, which we call bonding, and develop the properties of
bonding that we need for our work on positroids, and then we treat
those applications.  In the second, we prove other properties of
bonding that may help the future development of this topic.

\subsection{Bonding and its application  to positroids}
The notation that we introduce next is used throughout this section.
Let $M$ and $N$ be matroids for which $E(M)\cap E(N)$ is nonempty; let
$T$ denote this intersection.  We form the \emph{bonding of $M$ and
  $N$ at $T$}, denoted $B_T(M,N)$, via the following steps.
\begin{itemize}[leftmargin=*]
\item Say that $T$ is $\{t_1,t_2,\ldots,t_k\}$.  Fix sets
  $S=\{s_1,s_2,\ldots,s_k\}$ and $Q=\{q_1,q_2,\ldots,q_k\}$ that are
  disjoint from each other and from $E(M)$ and $E(N)$.
\item Form $N'$ from $N$ by, for each $i\in [k]$, relabeling $t_i$ as
  $s_i$.  Thus, $E(M)\cap E(N')=\emptyset$.
\item Extend the direct sum $M\oplus N'$ by, for each integer
  $i\in [k]$, adding $q_i$ freely to the flat
  $\cl_{M\oplus N'}(\{t_i,s_i\})$, giving the matroid $H$ on
  $E(M)\cup E(N')\cup Q$.
\item The bonding $B_T(M,N)$ is the minor $H/Q\del S$ of $H$.
\end{itemize}
Thus, $E(B_T(M,N))=E(M)\cup E(N)$. Note that $B_T(M,N)=B_T(N,M)$.
Also, the point $q_i$ is added freely to the line
$\cl_{M\oplus N'}(\{t_i,s_i\})$ of $M\oplus N'$ if $t_i$ is a loop of
neither $M$ nor $N$, while it is parallel to $s_i$ if $t_i$ is a loop
of $M$ only, and parallel to $t_i$ if $t_i$ is a loop of $N$ only, and
$q_i$ is a loop if $t_i$ is a loop of both $M$ and $N$.

\begin{figure}
  \centering
  \begin{tikzpicture}[scale=0.87]
    \draw[thick,
    black!30](6,-0.5)--(8,0.25)--(8,3.5)--(6,2.75)--(4,3.5)
    --(4,0.25)--(6,-0.5);%
    \draw[thick, black!30](6,1.15)--(6,2.75);%
    \draw[thick, black!30](6,0.85)--(6,0.75);%
    \draw[thick, black!30](6,-0.5)--(6,0.6);%

    \draw[very thick](2,0)--(2,3);%
    \filldraw (2,0) node[left=2] {\small$t_1$} circle (3pt);%
    \filldraw (2,1) node[left=2] {\small$t_2$} circle (3pt);%
    \filldraw (2,2) node[left=2] {\small$c$} circle (3pt);%
    \filldraw (2,3) node[left=2] {\small$d$} circle (3pt);%
    \node at (2,-1.1) {$N$};%

    \draw[very thick](0,0)--(0,3);%
    \filldraw (0.1,0) node[right=2] {\small$t_1$} circle (3pt);%
    \filldraw (-0.1,0) node[left=2] {\small$t_2$} circle (3pt);%
    \filldraw (0,1.5) node[right=2] {\small$a$} circle (3pt);%
    \filldraw (0,3) node[right=2] {\small$b$} circle (3pt);%
    \node at (0,-1.1) {$M$};%

    \draw[very thick](4.4,0.5)--(5.6,2.6);%
    \filldraw (4.4,0.5) node[below right] {\small$s_1$} circle
    (3pt);%
    \filldraw (4.8,1.2) node[left=1] {\small$s_2$} circle (3pt);%
    \filldraw (5.2,1.9) node[left=2] {\small$c$} circle (3pt);%
    \filldraw (5.6,2.6) node[left=2] {\small$d$} circle (3pt);%

    \draw[very thick](7,0.8)--(7,2.4);%
    \filldraw (7.1,0.8) node[right=2] {\small$t_1$} circle (3pt);%
    \filldraw (6.9,0.8) node[below =2] {\small$t_2$} circle
    (3pt);%
    \filldraw (7,1.6) node[right=2] {\small$a$} circle (3pt);%
    \filldraw (7,2.4) node[right=2] {\small$b$} circle (3pt);%

    \draw[very thick](7,0.8)--(4.4,0.5);%
    \draw[very thick](7,0.8)--(4.8,1.2);%
    \filldraw (5.7,0.65) node[below=2] {\small$q_1$} circle (3pt);%
    \filldraw (5.9,1) node[above left] {\small$q_2$} circle (3pt);%
     \node at (6,-1.1) {$H$};%

     \draw[very thick](10,0.25)--(10,2.45);%
     \filldraw (10.1,2.25) node[below right] {\small$t_1$} circle
     (3pt);%
     \filldraw (9.9,2.25) node[below left=-1] {\small$t_2$} circle
     (3pt);%
     \filldraw (9.9,2.45) node[above left] {\small$c$} circle (3pt);%
     \filldraw (10.1,2.45) node[above right] {\small$d$} circle
     (3pt);%
     \filldraw (10,0.25) node[right=2] {\small$a$} circle (3pt);%
     \filldraw (10,1.25) node[right=2] {\small$b$} circle (3pt);%
     \node at (10,-1.1) {$B_T(M,N)$};%
  \end{tikzpicture}
  \caption{Two rank-$2$ matroids $M$ and $N$, the matroid $H$, and the
    bonding $B_T(M,N)$, which is $H/\{q_1 ,q_2\}\del\{s_1,s_2\}$.  }
  \label{fig:bondex1}
\end{figure}

Figure \ref{fig:bondex1} gives a simple example that will show the
necessity of certain hypotheses in some results below.  Figure
\ref{fig:bondrev} gives an example that is more representative of what
we will see: if $T$ is independent in both $M$ and $N$, then
$B_T(M,N)$ is the free amalgam of $M$ and $N$; also, if, in addition,
$T$ is a set of clones, and $M$ and $N$ are positroids, then
$B_T(M,N)$ is a positroid.  That example shows that bonding two
transversal matroids (in fact, multi-path matroids) need not yield a
transversal matroid.  If $M$ and $N$ are the rank-$2$ uniform matroids
on $[q+1]$ and $[2q]-[q-1]$, respectively, so $T=\{q,q+1\}$, then
$B_T(M,N)$ is the rank-$2$ uniform matroid on $[2q]$.  Thus, bonding
need not preserve representability over a fixed finite field.
However, bonding preserves membership in classes of matroids that are
closed under direct sum, deletion, contraction, and principal
extension, such as the class of matroids that are representable over a
given infinite field.

\begin{figure}
  \centering
  \begin{tikzpicture}[scale=0.55]
    \draw[thick, black!30](16,-0.5)--(20,0.5)--(20,6.5)--(16,5.5)--(12,6.5)
    --(12,0.5)--(16,-0.5);%
    \draw[thick, black!30](16,-0.5)--(16,5.5);%

    \draw[very thick](0,2)--(3,0)--(3,4)--(0,2);%
    \filldraw (3,2.7) node[right=1] {\small$1$} circle (5pt);%
    \filldraw (3,1.3) node[right=1] {\small$2$} circle (5pt);%
    \filldraw (3,0) node[right=1] {\small$3$} circle (5pt);%
    \filldraw (1.5,1) node[below left] {\small$4$} circle (5pt);%
    \filldraw (0,2) node[left=1] {\small$5$} circle (5pt);%
    \filldraw (1.5,3) node[above left] {\small$6$} circle (5pt);%
    \filldraw (3,4) node[right=1] {\small$7$} circle (5pt);%

    \draw[very thick](6,0)--(6,4)--(9,2)--(6,0);%
    \filldraw (6,2.7) node[left =1] {\small$1$} circle (5pt);%
    \filldraw (6,1.3) node[left =1] {\small$2$} circle (5pt);%
    \filldraw (6,0) node[left =1] {\small$8$} circle (5pt);%
    \filldraw (7.5,1) node[below right] {\small$9$} circle (5pt);%
    \filldraw (9,2) node[right =1] {\small$10$} circle (5pt);%
    \filldraw (7.5,3) node[above right] {\small$11$} circle (5pt);%
    \filldraw (6,4) node[left =1] {\small$12$} circle (5pt);%

    \draw[very thick](16,0)--(16,5);%
    \draw[very thick](16,1)--(19,3)--(16,5);%
    \draw[very thick](16,0)--(13,2)--(16,4);%
    
    \filldraw (16,3) node[right =1] {\small$1$} circle (5pt);%
    \filldraw (16,2) node[right =1] {\small$2$} circle (5pt);%
    \filldraw (16,1) node[left =1] {\small$8$} circle (5pt);%
    \filldraw (16,0) node[right =1] {\small$3$} circle (5pt);%
    \filldraw (16,4) node[right =1] {\small$7$} circle (5pt);%
    \filldraw (16,5) node[left =1] {\small$12$} circle (5pt);%

    \filldraw (14.5,1) node[below left] {\small$4$} circle (5pt);%
    \filldraw (14.5,3) node[above left] {\small$6$} circle (5pt);%
    \filldraw (13,2) node[left =1] {\small$5$} circle (5pt);%
    \filldraw (17.5,2) node[below right] {\small$9$} circle (5pt);%
    \filldraw (19,3) node[right =1] {\small$10$} circle (5pt);%
    \filldraw (17.5,4) node[above right] {\small$11$} circle (5pt);%

    \node at (1.5,-1.5) {$M$};%
    \node at (7.5,-1.5) {$N$};%
    \node at (16,-1.5) {$B_T(M,N)$};%
  \end{tikzpicture}
  \caption{Two rank-$3$ positroids $M$ and $N$, and the bonding (their
    free amalgam) $B_T(M,N)$. Here, $T=\{1,2\}$ is an independent set
    of clones in $M$ and $N$.  The bonding is a rank-$4$ positroid.}
  \label{fig:bondrev}
\end{figure}

The way that bonding interacts with direct sums, which we treat in the
next lemma, is crucial in our work.

\begin{lemma}\label{lem:bds}
  Let $M$ be $M_1\oplus M_2\oplus \cdots \oplus M_s$ and $N$ be
  $N_1\oplus N_2\oplus \cdots \oplus N_t$.  Let $h\leq \min(s,t)$ and
  let $\{T_1,T_2,\ldots,T_h\}$ be a partition of $T$.  If, for each
  $i\in [h]$, we have $T_i=E(M_i)\cap E(N_i)$, then the bonding
  $B_T(M,N)$ is the direct sum of the bondings $B_{T_i}(M_i,N_i)$, for
  $i\in[h]$, along with $M_{h+1}, \ldots, M_s$ and
  $N_{h+1}, \ldots, N_t$.  Also, for $i$ with $h<i\leq s$, $E(M_i)$ is
  a connected component of $M$ if and only if $E(M_i)$ is a connected
  component of $B_T(M,N)$, and likewise for $E(N_j)$ with $h<j\leq t$.
\end{lemma}

\begin{proof}
  This follows from the construction since the matroid $H$ that is
  used to construct $B_T(M,N)$ is the direct sum of its counterparts
  in the construction of $B_{T_i}(M_i,N_i)$, for $i\in[h]$, with the
  other direct sum factors $M_{h+1}, \ldots, M_s$ and
  $N_{h+1}, \ldots, N_t$.
\end{proof}

We next relate the flats of $B_T(M,N)$ to those of $H$.

\begin{lemma}\label{lem:bflat}
  Let $F$ be a flat of $B_T(M,N)$, and set
  $S_F=\{s_i\,:\,t_i\in F\cap T\}$ and
  $Q_F=\{q_i\,:\,t_i\in F\cap T\}$.  Then $F\cup Q\cup S_F$ is a flat
  of $H$, each $q_j\in Q-Q_F$ is a coloop of $H|(F\cup Q\cup S_F)$,
  and $B_T(M,N)|F= H|(F\cup Q_F\cup S_F)/Q_F\del S_F$.
\end{lemma}

\begin{proof}
  Since $F$ is a flat of $B_T(M,N)$, which is $H\del S/Q$, the set
  $F\cup Q$ is a flat of $H\del S$.  Its closure $\cl_H(F\cup Q)$ in
  $H$ is $F\cup Q\cup S_F$ since any two of $t_i$, $s_i$, and $q_i$
  have the third in their closure.  If $q_j\in Q-Q_F$, then $q_j$ is a
  coloop of $H|(F\cup Q\cup S_F)$ since $q_j$ was added freely to the
  flat spanned by $t_j$ and $s_j$, which are not in the flat
  $F\cup Q\cup S_F$.  Thus,
  $$B_T(M,N)|F=H|(F\cup Q\cup S_F)/Q\del S_F= H|(F\cup Q_F\cup
  S_F)/Q_F\del S_F.\qedhere$$
\end{proof}

The next lemma allows us to deduce that certain flats of $B_T(M,N)$
are disconnected.  In this lemma and throughout this section, for
$X\subseteq E(M)\cup E(N)$, we shorten $X\cap E(M)$ to $X_M$, and
likewise set $X_N=X\cap E(N)$.

\begin{lemma}\label{lem:discflats}
  Let $F$ be a flat of $B_T(M,N)$.  If there are separators $X$ of
  $M|F_M$ and $Y$ of $N|F_N$, not both trivial, for which
  $X\cap T=Y\cap T$, then $X\cup Y$ is a nontrivial separator of
  $B_T(M,N)|F$, so $B_T(M,N)|F$ is disconnected.
\end{lemma}

\begin{proof}
  We have $B_T(M,N)|F= H|(F\cup Q_F\cup S_F)/Q_F\del S_F$ where $S_F$
  and $Q_F$ are as in Lemma \ref{lem:bflat}.  By symmetry, we may
  assume that the separator $X$ of $M|F_M$ is neither $\emptyset$ nor
  $F_M$.  Set $T'= X\cap T=Y\cap T$, and let
  $S_{T'}=\{s_i\,:\,t_i\in T'\}$ and $Q_{T'}=\{q_i\,:\,t_i\in T'\}$.
  Thus, $H|(F\cup Q_F\cup S_F)$ is the direct sum
  $$(H|(X\cup Y\cup Q_{T'}\cup S_{T'}))\oplus(H |((F-(X\cup Y))\cup
  (Q_F-Q_{T'})\cup (S_F-S_{T'}))),$$ and neither $X\cup Y$ nor
  $F-(X\cup Y)$ is empty.  Therefore $X\cup Y$ is a nontrivial
  separator of $H|(F\cup Q_F\cup S_F)/Q_F\del S_F$, that is, of
  $B_T(M,N)|F$.
\end{proof}

We next show that bonding and restriction commute when restricting to
supersets of $T$.

\begin{lemma}\label{lem:brest}
  If $T\subseteq X\subseteq E(M)\cup E(N)$, then
  $B_T(M,N)|X= B_T(M|X_M,N|X_N)$.
\end{lemma}

\begin{proof}
  Now $B_T(M,N)|X = H/Q\del S|X=H|(X\cup Q\cup S)/Q\del S$ since
  deletion and contraction commute.  With that, the lemma follows by
  showing that $H|(X\cup Q\cup S)$ plays the role of $H$ in the
  construction of $B_T(M|X_M,N|X_N)$.  That holds since (a)
  $H|(X\cup S)$ is $(M|X_M)\oplus (N'|X_{N'})$ where
  $X_{N'}= (X_N-T)\cup S$ and (b) for each $t_i\in T$, the point $q_i$
  is added freely to $\cl_{M\oplus N'}(\{t_i,s_i\})$ in $M\oplus N'$
  to get $H$, so $q_i$ is added freely to
  $\cl_{H|(X\cup S)}(\{t_i,s_i\})$ in $H|(X\cup S)$ to get
  $H|(X\cup Q\cup S)$.
\end{proof}

We next treat the counterpart for contraction.

\begin{lemma}\label{lem:contP}
  If  
  $T\subseteq X\subseteq E(M)\cup E(N)$, then
  $B_T(M,N)/X=(M/X_M)\oplus(N/X_N)$.
\end{lemma}

\begin{proof}
  We show the special case $B_T(M,N)/T=(M/T)\oplus(N/T)$, from which
  the general result follows.  Observe that
  $\cl_H(S\cup T)=\cl_H(S\cup Q)=\cl_H(Q\cup T)$.  Thus,
  \begin{align*}
    B_T(M,N)/T 
    & = H/(Q\cup T)\del S\\
    & = H/(S\cup T)\del Q\\
    & = H\del Q/(S\cup T)\\
    & = (M/T)\oplus (N'/S)\\
    &= (M/T)\oplus (N/T).\qedhere
  \end{align*}
\end{proof}

We next treat a key contraction property for bonding that shows the
traction we gain by using bonding.  When $T$ is independent in both
$M$ and $N$, as we will see, $B_T(M,N)$ is the free amalgam of $M$ and
$N$, but the bonding $B_T (M/X,N/Y)$ in the next lemma need not be an
amalgam of $M/X$ and $N/Y$; indeed, we may have $M/X|T\ne N/Y|T$.

\begin{lemma}\label{lem:contrinB}
  If $X\subseteq E(M)-T$ and $Y\subseteq E(N)-T$, then
  $$B_T(M,N)/(X\cup Y) = B_T (M/X,N/Y).$$
\end{lemma}

\begin{proof}
  This holds since deletion and contraction commute:
  \begin{align*}
    B_T(M,N)/(X\cup Y)
    & =H/Q\del S/(X\cup Y)\\
    & =H/(X\cup Y)/Q\del S\\
    & = B_T(M/X,N/Y)
  \end{align*}
  since, to construct $ B_T(M/X,N/Y)$, the matroid $H/(X\cup Y)$ plays
  the role of $H$.  
\end{proof}

The next lemma relates $B_T(M,N)/P$ and $B_{T-P} (M/P,N/P)$ when
$\emptyset\ne P\subsetneq T$.

\begin{lemma}\label{lem:contractinP}
  If $\emptyset\ne P\subsetneq T$, then
  $B_T(M,N)/P = B_{T-P} (M/P,N/P)$.
\end{lemma}

\begin{proof}
  It suffices to prove that $B_T(M,N)/t_1 = B_{T-t_1} (M/t_1,N/t_1)$
  for any $t_1\in T$, so we focus on that.  Since deletion and
  contraction commute,
  $$B_T(M,N)/t_1=H/Q\del S/t_1=H/q_1 \del s_1
  /t_1/(Q-q_1)\del (S-s_1).$$ Thus, it suffices to show that
  $H/q_1\del s_1/t_1$ is the counterpart of $H$ in the construction of
  $B_{T-t_1} (M/t_1,N/t_1)$, that is, it is the direct sum of $M/t_1$
  and $N'/s_1$ (the copy of $N/t_1$ where $T-t_1$ is replaced by
  $S-s_1$), to which, for $i$ with $2\leq i\leq k$, the element $q_i$
  is added freely to the flat
  $\cl_{(M/t_1)\oplus (N'/s_1)}(\{t_i,s_i\})$.  For a set
  $X\subseteq E(H)-\{q_1, s_1,t_1,q_i\}$ where $2\leq i\leq k$, the
  following statements are equivalent:
  \begin{itemize}
  \item $q_i\in\cl_{H/q_1\del s_1/t_1}(X)$,
  \item $q_i\in\cl_H(X\cup \{q_1,t_1\})$,
  \item $t_i,s_i\in\cl_H(X\cup \{q_1,t_1\})$,
  \item $t_i,s_i\in\cl_{H/q_1\del s_1/t_1}(X)$.
  \end{itemize}
  Thus, $q_i$ is added freely to the flat
  $\cl_{(M/t_1)\oplus (N'/s_1)}(\{t_i,s_i\})$ of
  $(M/t_1)\oplus (N'/s_1)$.  Since
  $r_H(\{q_1,t_1\}) = r_H(\{s_1,t_1\})$, we have
  \begin{align*}
    r(H/q_1\del s_1/t_1)
    & = r(H)-r_H(\{q_1,t_1\})\\
    & =r(M)+r(N')- r_H(\{s_1,t_1\})\\
    & = r(M/t_1)+r(N'/s_1).
  \end{align*}
  If $r_M(t_1)=1=r_N(t_1)$, then $q_1$ is in the span of neither
  $E(M)$ nor $E(N')$ in $H$, so $M$ and $N'$ are restrictions of
  $H/q_1$, and so, since $t_1$ and $s_1$ are parallel in $H/q_1$, both
  $M/t_1$ and $N'/s_1$ are restrictions of $H/q_1\del s_1/t_1$. If
  $r_M(t_1)=0$ and $r_N(t_1)=1$, then $q_1$ and $s_1$ are parallel in
  $H$, so $H/q_1\del s_1=H\del q_1/s_1$, and so, with $t_1$ being a
  loop, $N'/s_1$ is a restriction of $H/q_1\del s_1/t_1$, as is
  $M/t_1$.  The other two options for $r_M(t_1)$ and $r_N(t_1)$
  similarly imply that $M/t_1$ and $N'/s_1$ are restrictions of
  $H/q_1\del s_1/t_1$. So, to complete the proof, observe that with
  the rank equality above, we get
  $H/q_1\del s_1/t_1\del(Q-q_1)= (M/t_1)\oplus(N'/s_1)$.
\end{proof}
  
By the next lemma, $B_T(M,N)$ is an amalgam of $M$ and $N$ if $T$ is
independent in both $M$ and $N$.  Figures \ref{fig:bondex1} (with $T$
dependent) and \ref{fig:bondrev} (with $T$ independent) illustrate the
lemma.

\begin{lemma}\label{lem:indeprest}
  If $T$ is independent in $M$, then the restriction $B_T(M,N)|E(N)$
  is $N$.  If $T$ is dependent in $M$ and no element of $T$ is a loop
  of $N$, then $B_T(M,N)|E(N)$ is a proper quotient of $N$.
\end{lemma}

\begin{proof}
  Let $T_0$ be a basis of $M|T$.  Set $Q_0=\{q_i\,:\,t_i\in T_0\}$. For
  all $j$, the sets $\{s_j,q_j\}$ and $\{s_j,t_j\}$ have the same
  closure in $H$, so the four sets
  $$E(N')\cup Q,\quad E(N')\cup T, \quad E(N')\cup T_0, \quad
  \text{and} \quad E(N')\cup Q_0$$ have the same closure in $H$.  The
  rank of that closure is
  $$r_H(E(N')\cup Q)=r(N)+r_M(T)=r(N)+|T_0|=r(N)+|Q_0|.$$ Thus, $Q_0$ is
  independent in $H$ and $H/Q_0|E(N')=N'$.
     
  If $T$ is independent in $M$, then $T_0=T$, so $Q_0=Q$, and so
  $H/Q|E(N')$ is $N'$.  The first assertion follows since $t_i$ and
  $s_i$ are either parallel in $H/Q$ or, if $t_i$ is a loop of $N$,
  then $t_i$ and $q_i$ are parallel in $H$, so $t_i$ is a loop of
  $H/Q$.

  Now assume that $T$ is dependent in $M$ and contains no loops of
  $N$.  Since $|T_0|<|T|$, we have $Q-Q_0\ne \emptyset$, so
  $H|(E(N')\cup Q)/Q_0$ is a rank-$r(N)$ extension of $N'$.  If
  $q_j\in Q-Q_0$, then $s_j\not\in \cl_H(Q_0)$, for otherwise there
  would be a circuit $C$ so that $s_j\in C\subseteq Q_0\cup s_j$,
  which would give $q_i\in \cl_H(E(N')\cup (Q_0-q_i))$ for any
  $q_i\in C$, contrary to the equality
  $r(E(N')\cup Q_0)= r(N')+|Q_0|$; so $q_j\not\in\cl_H(Q_0)$ since
  $q_j$ was added freely to the flat spanned by $s_j$ and $t_j$.
  Thus, $\cl_H(Q_0)\cap (Q-Q_0)=\emptyset$, so $Q-Q_0$ has positive
  rank in the extension $H|(E(N')\cup Q)/Q_0$ of $N'$.  To get
  $B_T(M,N)|E(N)$ from that extension of $N'$, besides the relabeling
  relating $N$ and $N'$, we contract $Q-Q_0$, so $B_T(M,N)|E(N)$ is a
  proper quotient of $N$.
\end{proof}

Thus, in general, $B_T(M,N)$ is an amalgam of quotients of $M$ and
$N$.

\begin{cor}\label{cor:flatscomefromflats}
  Assume that the set $T$ is independent in both $M$ and $N$.  If $F$
  is a flat of $B_T(M,N)$, then $F_M$ is a flat of $M$ and $F_N$ is a
  flat of $N$.  If both $M$ and $N$ are connected, then so is
  $B_T(M,N)$.
\end{cor}

\begin{proof}
  Lemma \ref{lem:indeprest} gives the result about flats.  For
  distinct $e,f\in E(M)\cup E(N)$, if $T\cap\{e,f\}\ne\emptyset$, then
  $e$ and $f$ are in a circuit of either $M$ or $N$; otherwise each of
  $e$ and $f$ is in a circuit of either $M$ or $N$ with some $t\in T$.
  Such circuits are circuits of $B_T(M,N)$, so $e$ and $f$ are in the
  same component of $B_T(M,N)$.  Thus, $B_T(M,N)$ is connected.
\end{proof}

If $T$ is dependent in both $M$ and $N$, then $B_T(M,N)$ can be
disconnected even if both $M$ and $N$ are connected; see Figure
\ref{fig:discbond} for an example.  Also, $B_T(M,N)$ may be connected
even if both $M$ and $N$ are disconnected and $T$ is independent in
both: this occurs, for instance, when $M$ is the direct sum of a
$4$-circuit on $\{1,2,a,b\}$ and a coloop $3$, while $N$ is the direct
sum of a $4$-circuit on $\{2,3,c,d\}$ and a coloop $1$. Thus, when $F$
is a connected flat of $B_T(M,N)$, without additional hypotheses (as
in Lemmas \ref{lem:conrestP} and \ref{lem:conrestP2} below), we cannot
deduce that either $F_M$ is connected in $M$ or $F_N$ is connected in
$N$.

\begin{figure}
  \centering
  \begin{tikzpicture}[scale=0.9]
    \draw[very thick](0,0)--(1.2,1.5);%
    \draw[very thick](0,0)--(-1.2,1.5);%
    \filldraw (-0.1,0) node[left=1] {\small$1$} circle (2.75pt);%
    \filldraw (0.1,0) node[right=1] {\small$2$} circle (2.75pt);%
    \filldraw (-0.6,0.75) node[left =1] {\small$a$} circle (2.75pt);%
    \filldraw (-1.2,1.5) node[left=1] {\small$b$} circle (2.75pt);%
    \filldraw (0.6,0.75) node[right=1] {\small$c$} circle (2.75pt);%
    \filldraw (1.2,1.5) node[right =1] {\small$d$} circle (2.75pt);%
    \node at (0,-0.5) {$M$};%

    \draw[very thick](4,0)--(2.8,1.5);%
    \draw[very thick](4,0)--(5.2,1.5);%
    \filldraw (3.9,0) node[left=1] {\small$1$} circle (2.75pt);%
    \filldraw (4.1,0) node[right=1] {\small$2$} circle (2.75pt);%
    \filldraw (3.4,0.75) node[left =1] {\small$e$} circle (2.75pt);%
    \filldraw (2.8,1.5) node[left=1] {\small$f$} circle (2.75pt);%
    \filldraw (4.6,0.75) node[right=1] {\small$g$} circle (2.75pt);%
    \filldraw (5.2,1.5) node[right =1] {\small$h$} circle (2.75pt);%
    \node at (4,-0.5) {$N$};%
  \end{tikzpicture}
  \caption{Two connected rank-$3$ matroids whose bonding at the
    circuit $T=\{1,2\}$ is disconnected; $B_T(M,N)$ has six connected
    components, namely, $\{a,b\}$, $\{c,d\}$, $\{e,f\}$, $\{g,h\}$,
    $\{1\}$, and $\{2\}$.  }
  \label{fig:discbond}
\end{figure}

The next result gives the ranks of the flats of $B_T(M,N)$.

\begin{lemma}\label{lem:flatrankinB}
  Assume that $T$ is independent in $M$.  The rank $r(F)$ of a flat
  $F$ of $B_T(M,N)$ is $r_M(F_M)+r_N(F_N)-|F\cap T|$.
\end{lemma}

\begin{proof}
  Lemma \ref{lem:bflat} gives
  $B_T(M,N)|F= H|(F\cup Q_F\cup S_F)/Q_F\del S_F$ where $Q_F$ and
  $S_F$ are defined in that lemma.  Since
  $\cl_H(F\cup Q_F)=\cl_H(F\cup S_F) $, we have
  $$r(F) =r(B_T(M,N)|F) = r_H(F\cup S_F)-r_H(Q_F).$$ By how $H$ is
  constructed, $H|(F\cup S_F)$ is the direct sum of $M|F_M$ and the
  restriction of $N'$ that is a relabeling of $N|F_N$.  As shown in
  the proof of Lemma \ref{lem:indeprest}, having $T$ independent in
  $M$ implies that $Q$ is independent in $H$.  Thus,
  $r(F) = r_M(F_M)+r_N(F_N)-|Q_F|$. The equality in the lemma follows
  since $|Q_F|=|T\cap F|$.
\end{proof}

This lemma and Theorem \ref{thm:free} have the following corollary.

\begin{cor}\label{cor:famal}
  If $T$ is independent in both $M$ and $N$, then $B_T(M,N)$ is the
  free amalgam of $M$ and $N$.
\end{cor}

The terminology of free amalgams is well established in the
literature, so we will use that name when $T$ is independent in both
$M$ and $N$, but, noting Lemma \ref{lem:contrinB} and the remarks
before it, we cannot work solely with free amalgams; we sometimes need
the more general setting of bonding.

The next result follows from Corollary \ref{cor:famal} when
$r_M(p)=1=r_N(p)$.  The other cases are easy to verify directly from the
definitions of $B_p(M,N)$ and $P(M,N)$.

\begin{cor}\label{cor:scbpc}
  If $E(M)\cap E(N)=\{p\}$, then $B_p(M,N)$ is the parallel connection
  $P(M,N)$ of $M$ and $N$ at the base point $p$.
\end{cor}

We are most interested in connected flats.  The next corollary narrows
the collection of flats that we must consider.

\begin{cor}\label{cor:discflat}
  Assume that $T$ is independent in both $M$ and $N$.  If $F$ is a
  flat of $B_T(M,N)$ with $F\cap T=\emptyset$ and neither $F_M$ nor
  $F_N$ is empty, then $F$ is disconnected.
\end{cor}

The bonding $B_T(M,N)$ has additional useful properties when $T$ is a
set of clones in both $M$ and $N$.  We first show that clones in both
$M$ and $N$ are clones in $B_T(M,N)$.

\begin{lemma}
  If $t_i,t_j\in T$ are clones in both $M$ and $N$, then they are
  clones in $B_T(M,N)$. Thus, if $T$ is a set of clones in both $M$
  and $N$, then $T$ is a set of clones in $B_T(M,N)$.
\end{lemma}

\begin{proof}
  By symmetry, it suffices to show that if $F$ is a cyclic flat of
  $B_T(M,N)$ and $t_i\in F$, then $t_j\in F$.  Set
  $S_F=\{s_h\,:\,t_h\in F\cap T\}$. By Lemma \ref{lem:bflat}, the set
  $F\cup Q\cup S_F$ is a flat of $H$.  At least one of $t_i$ and $s_i$
  is not a coloop of $H|(F\cup S_F)$, for otherwise $t_i$ would be a
  coloop of both $M|F_M$ and $N|F_N$, so $t_i$ would be a coloop of
  $B_T(M,N)|F$ by Lemma \ref{lem:discflats}, contrary to the flat $F$
  being cyclic.  If $t_i$ is not a coloop of $H|(F\cup S_F)$, then its
  clone $t_j$ in $M$ must be in the flat $F\cup Q\cup S_F$, so
  $t_j\in F$; otherwise $s_i$ is not a coloop of $H|(F\cup S_F)$, so
  its clone $s_j$ must be in the flat $F\cup Q\cup S_F$, and so
  $t_j\in F$.
\end{proof}
  
Since two elements are clones in $M$ if and only if they are in the
same nonsingleton connected flats of $M$, we get the next result.

\begin{lemma}\label{lem:BTClone}
  If $T$ is a set of clones in both $M$ and $N$, and if $F$ is a
  connected flat of $B_T(M,N)$ with $|F|\geq 2$, then either
  $T\subseteq F$ or $F\cap T=\emptyset$.
\end{lemma}

We have seen that, when $T$ is independent in $M$ and $N$, if $F$ is a
connected flat of $B_T(M,N)$ and $F\cap T=\emptyset$, then $F$ is a
connected flat of one of $M$ or $N$.  The next lemma treats the case
of $T\subsetneq F$, assuming that $T$ is a set of clones.  One can
give a shorter proof of this lemma in the case of a connected flat by
applying Lemma \ref{lem:discflats} (and that would suffice for our
applications of the next lemma), but the next lemma treats connected
sets, not just connected flats.  Note also that the lemma does not
assume that $T$ is independent in $M$ and $N$, so $M$ and $N$ might
not be restrictions of $B_T(M,N)$.

\begin{lemma}\label{lem:conrestP}
  Assume that $T$ is a set of clones in $M$ and in $N$.  Let $X$ be a
  connected set in $B_T(M,N)$ with $T\subsetneq X$.  If $X_M\ne T$,
  then $M|X_M$ is connected, and likewise for $N|X_N$ if $X_N\ne T$.
\end{lemma}

\begin{proof}
  By symmetry, it suffices to prove the result for $M|X_M$ when
  $X_M\ne T$.  Lemma \ref{lem:brest} gives
  $B_T(M,N)|X = B_T(M|X_M,N|X_N)$. Since $T$ is a set of clones in
  $M$, either (i) $|T|\geq 2$ and $T$ contains only loops or only
  coloops of $M|X_M$, or (ii) $T\subseteq K$ for some connected
  component $K$ of $M|X_M$.  Since $T\ne X_M$, if option (i) held,
  then $M|X_M$ would have more than $|T|$ connected components, so
  Lemmas \ref{lem:brest} and \ref{lem:bds} would give the
  contradiction that $B_T(M,N)|X$, which is connected, is a direct sum
  of at least two matroids.  So option (ii) holds.  By Lemmas
  \ref{lem:brest} and \ref{lem:bds}, if $M|X_M$ had any connected
  components besides $K$, then $B_T(M,N)|X$ would be a direct sum of
  at least two matroids, but $B_T(M,N)|X$ is connected. Thus, $M|X_M$ is
  connected.
\end{proof}

We turn to the applications to positroids.

\begin{thm}\label{thm:bondmain1}
  Let $M$ and $N$ be positroids with no loops, and let $E(M)\cap E(N)$
  be $T$.  If $T$ is nonempty, $T$ is independent in both $M$ and $N$,
  and $T$ is a set of clones in both $M$ and $N$, then the free
  amalgam $B_T(M,N)$ is a positroid.
\end{thm}

\begin{proof}
  Let $T=\{t_1,t_2,\ldots,t_k\}$.  Observe that if the elements of $T$
  are coloops of $M$, then $B_T(M,N)=(M\del T)\oplus N$, in which case
  the result follows since the class of positroids is closed under
  deletion and direct sum.  Thus, by symmetry and the assumption that
  $T$ is a set of clones of both $M$ and $N$, we may assume that no
  element of $T$ is a coloop of either $M$ or $N$.  By Lemma
  \ref{lem:autclone} and Corollary \ref{cor:clone}, some positroid
  order for $M$ and some positroid order for $N$ have the same
  restriction to $T$ and have $T$ being an interval.  Let
  $$e_1<_Me_2<_M\cdots<_M e_m<_Mt_1<_Mt_2<_M\cdots <_Mt_k$$
  and
  $$t_1<_Nt_2<_N\cdots <_Nt_k<_Nf_1<_Nf_2<_N\cdots<_N f_n$$
  be such positroid orders for $M$ and $N$, respectively.  We claim
  that the linear order
  $$e_1<e_2<\cdots< e_m<t_1<t_2<\cdots <t_k <f_1<f_2<\cdots< f_n$$
  of $E(B_T(M,N))$ satisfies the cyclic interval property, so
  $B_T(M,N)$ is a positroid.  Note that each of $\leq_M$ and $\leq_N$
  is a restriction of $\leq$ to an interval.  The following
  observations about cyclic intervals in these orders are used below.
  \begin{itemize}
  \item[(I$_1$)] Assume that $F\subseteq E(M)$ and $F\cap T=\emptyset$,
    and that $I$ is a cyclic interval in $\leq_M$ with
    $I\cap F=\emptyset$.  If $I\cap T=\emptyset$, then $I$ is also a
    cyclic interval in $\leq$; if $I\cap T\ne\emptyset$, then
    $I\cup E(N)$ is a cyclic interval in $\leq$ and it is disjoint
    from $F$.  The same statement holds with the roles of $M$ and $N$
    reversed.
  \item[(I$_2$)] If $T\subseteq F\subseteq E(M)\cup E(N)$, then any
    cyclic interval in $\leq_M$ that is disjoint from $F$ is a cyclic
    interval in $\leq$, as is any cyclic interval in $\leq_N$ that is
    disjoint from $F$.
  \end{itemize}

  Let $F$ be a connected flat of $B_T(M,N)$ with $|F|\geq 2$.  By the
  hypotheses about clones and Lemma \ref{lem:BTClone}, either
  $F\cap T=\emptyset$ or $T\subsetneq F$; we address each case below.

  First assume that $F\cap T=\emptyset$.  By Corollaries
  \ref{cor:flatscomefromflats} and \ref{cor:discflat}, the set $F$ is
  a connected flat of either $M$ or $N$, say of $M$.  No element of
  $T$ is a coloop of $M$, so no element of $T$ is a coloop of $M/F$ by
  Lemma \ref{lem:cocofcon}, and so all elements of $T$, which are
  clones in $M/F$, are in the same connected component, say $X$, of
  $M/F$.  Lemma \ref{lem:contrinB} gives
  $B_T(M,N)/F=B_T(M/F,N)$. (Unless $T$ is independent in $M/F$, there
  is no amalgam of $M/F$ and $N$, but their bonding is defined.  While
  $M/F$ is a restriction of $B_T(M/F,N)$, only when $T$ is independent
  in $M/F$ will $N$ be a restriction of $B_T(M/F,N)$.)  Let $K$ be a
  connected component of $B_T(M,N)/F$ with $|K|\geq 2$.  Since $T$ is
  a set of clones of $B_T(M,N)$ and so of $B_T(M,N)/F$, either
  $T\subseteq K$ or $T\cap K=\emptyset$.  Now either
  \begin{itemize}
  \item[(a)] $K\subseteq E(N)$,
  \end{itemize}
  or, by applying Lemma \ref{lem:bds} to $B_T(M/F,N)$, one of the
  following statements holds:
  \begin{itemize}
  \item[(b)] $K$ is a connected component of $M/F$ for which
    $K\cap T=\emptyset$, or
  \item[(c)] $K$ is a connected component of the bonding
    $B_T(M/F|X,N|Y)$ for some set $Y$ with
    $T\subseteq Y\subseteq E(N)$.
  \end{itemize}
  The cyclic interval property holds in case (a) since $E(N)$ is a
  cyclic interval in $\leq$.  In case (b), there is a cyclic interval
  $I$ in $\leq_M$ that contains $K$ and has $I\cap F=\emptyset$, so
  observation (I$_1$) above shows that the cyclic interval property
  holds.  In case (c), for any cyclic interval $I$ in $\leq_M$ that
  contains $X$ and has $I\cap F=\emptyset$, by observation (I$_1$),
  the cyclic interval $I\cup E(N)$ in $\leq$ shows that the cyclic
  interval property holds.
  
  Now assume that $T\subsetneq F$.  By Lemma \ref{lem:contP}, we have
  $B_T(M,N)/T=(M/T)\oplus (N/T)$.  If $F\subseteq E(M)$, then $F$ is a
  connected flat of $M$ and $B_T(M,N)/F=(M/F)\oplus (N/T)$, so each
  connected component $K$ of $B_T(M,N)/F$ is a connected component of
  either $M/F$ or $N/T$; thus, $K$ is contained either in a cyclic
  interval in $\leq_M$ that is disjoint from $F$, or in the interval
  $E(N)-T$; such cyclic intervals are cyclic intervals in $\leq$, so
  the cyclic interval property holds.  The same idea applies if
  $F\subseteq E(N)$.  Now assume that $F\not\subseteq E(M)$ and
  $F\not\subseteq E(N)$.  By Corollary \ref{cor:flatscomefromflats}
  and Lemma \ref{lem:conrestP}, the sets $F_M$ and $F_N$ are connected
  flats of $M$ and $N$, respectively.  Also,
  $B_T(M,N)/F=(M/F_M)\oplus (N/F_N)$. Thus, each connected component
  of $B_T(M,N)/F$ is a connected component of one of $M/F_M$ and
  $N/F_N$, and the cyclic interval property holds since it holds in
  $M$ and $N$.
\end{proof}

With the assumptions above, while $B_T(M,N)$ is a positroid, the
matroid $H$ that we use to construct $B_T(M,N)$ need not be a
positroid.  For instance, $H$ is not a positroid when $M$ and $N$ are
both the uniform matroid $U_{3,4}$ and $|T|=3$.
 
Corollary \ref{cor:scbpc} and Theorem \ref{thm:bondmain1} give one
case of the next result on parallel connection; the other case, when
$p$ is a loop of one or both of $M$ and $N$, holds since contractions
and direct sums of positroids are positroids.  Recall that duals of
positroids are positroids, and series connections are dual to parallel
connections: $S(M,N) = \bigl(P(M^*,N^*)\bigr)^*$.  Blum
\cite[Corollary 4.2]{blum} treated the special cases of parallel and
series extension.

\begin{cor}\label{cor:parcon}
  If $M$ and $N$ are positroids and $|E(M)\cap E(N)|=1$, then the
  parallel connection $P(M,N)$ and the series connection $S(M,N)$ are
  positroids.
\end{cor}

A \emph{series-parallel network} is a graph that can be obtained from
a graph with a single edge (perhaps a loop) by repeatedly applying
parallel extension and series extension.  Duffin \cite{duffin} proved
that a $2$-connected graph is a series-parallel network if and only if
it has no $K_4$-minor.  Binary matroids with no $M(K_4)$ minor are
graphic since the excluded minors for graphic matroids, other than the
excluded minor for binary matroids, have $M(K_4)$ as a minor. (See
Tutte \cite{tutte} for the excluded minors.)  So a matroid is the
cycle matroid of a series-parallel network if and only if it is
connected, binary, and has no $M(K_4)$-minor.  The next corollary,
which is also \cite[Theorem 5.1]{blum}, follows since the class of
positroids is closed under parallel extension and series extension,
and $M(K_4)$ is not a positroid.

\begin{cor}
  A binary matroid is a positroid if and only if the restriction to
  each of its connected components is the cycle matroid of a
  series-parallel network.
\end{cor}

In particular, cycle matroids of series-parallel networks are positroids.

\begin{figure}
  \centering
  \begin{tikzpicture}[scale=0.55]
    \draw[thick, black!30](16,-0.5)--(20,0.5)--(20,6.25)--(16,5.25)--(12,6.25)
    --(12,0.5)--(16,-0.5);%
    \draw[thick, black!30](16,-0.5)--(16,5.25);%

    \draw[very thick](0,2)--(3,0)--(3,4)--(0,2);%
    \filldraw (3,2) node[right=1] {\small$10$} circle (5pt);%
    \filldraw (3,0) node[right=1] {\small$5$} circle (5pt);%
    \filldraw (1.5,1) node[below left] {\small$4$} circle (5pt);%
    \filldraw (0,2) node[left=1] {\small$3$} circle (5pt);%
    \filldraw (1.5,3) node[above left] {\small$2$} circle (5pt);%
    \filldraw (3,4) node[right=1] {\small$1$} circle (5pt);%

    \draw[very thick](6,0)--(6,4)--(9,2)--(6,0);%
    \filldraw (6,2) node[left =1] {\small$10$} circle (5pt);%
    \filldraw (6,0) node[left =1] {\small$5$} circle (5pt);%
    \filldraw (7.5,1) node[below right] {\small$6$} circle (5pt);%
    \filldraw (9,2) node[right =1] {\small$7$} circle (5pt);%
    \filldraw (7.5,3) node[above right] {\small$8$} circle (5pt);%
    \filldraw (6,4) node[left =1] {\small$9$} circle (5pt);%

    \draw[very thick](16,0.5)--(16,4.5);%
    \draw[very thick](16,0.5)--(19,3)--(16,4.5);%
    \draw[very thick](16,0.5)--(13,2)--(16,3.5);%
    
    \filldraw (16,2) node[right =1] {\small$10$} circle (5pt);%
    \filldraw (16,0.5) node[below left =-1] {\small$5$} circle (5pt);%
    \filldraw (16,3.5) node[right =1] {\small$1$} circle (5pt);%
    \filldraw (16,4.5) node[left =1] {\small$9$} circle (5pt);%
    
    \filldraw (14.5,1.25) node[below left] {\small$4$} circle (5pt);%
    \filldraw (14.5,2.75) node[above left] {\small$2$} circle (5pt);%
    \filldraw (13,2) node[left =1] {\small$3$} circle (5pt);%
    \filldraw (17.5,1.75) node[below right] {\small$6$} circle (5pt);%
    \filldraw (19,3) node[right =1] {\small$7$} circle (5pt);%
    \filldraw (17.5,3.75) node[above right] {\small$8$} circle (5pt);%

    \node at (1.5,-1.5) {$M$};%
    \node at (7.5,-1.5) {$N$};%
    \node at (16,-1.5) {$B_T(M,N)$};%
  \end{tikzpicture}
  \caption{The free amalgam $B_T(M,N)$ of the rank-$3$ positroids $M$
    and $N$, where $T=\{5,10\}$.  The free amalgam is a positroid, but
    $5$ and $10$ are not clones.  The set $P$ in the hypothesis of
    Theorem \ref{thm:bondmain2} is $\{5\}$.}
  \label{fig:bondrevvar}
\end{figure}

In the proof of Theorem \ref{thm:bondmain1}, to get a linear order $<$
that extends both $<_M$ and $<_N$, we can start with $E(M)-T$ as an
interval followed by $E(N)-T$, and then put $T$ either between them
(as we did) or after $E(N)-T$.  Having two options suggests that a
similar argument may work if $T$ is the union of two sets of clones;
we may be able to put one set of clones between $E(M)-T$ and $E(N)-T$,
and the other after $E(N)-T$.  Theorem \ref{thm:bondmain2} below,
which we illustrate in Figure \ref{fig:bondrevvar}, treats that.  The
conditions on $T$ in Theorem \ref{thm:bondmain2} are stronger than
those in Theorem \ref{thm:bondmain1} in that $M|\cl_M(T)$ and
$N|\cl_N(T)$ are required to be connected, but weaker in that, instead
of requiring $T$ to be a set of clones, $T$ is partitioned into two
sets of clones, $P$ and $T-P$, where all nonsingleton connected flats
of either $M$ or $N$ that contain at least one element of $T-P$
contain all of $T$.

Many positroids satisfy the hypotheses of Theorem \ref{thm:bondmain2}.
For instance, start with an $n$-spoke wheel, delete at most $n-2$
spokes (see Figure \ref{fig:modifiedwhirl}), take the cycle matroid,
and relax the circuit-hyperplane of rim edges.  Fix two consecutive
spokes that remain, let $P$ contain one of those spokes, and let $T-P$
be the set of rim edges between those spokes (e.g., $T$ can be the set
of the thick edges in Figure \ref{fig:modifiedwhirl}).  The set $T$ is
independent, its closure is connected (it is the circuit consisting of
the two consecutive spokes and the rim edges between them), and all
nonsingleton connected flats that contain one of those rim edges
contain that circuit (relaxing the circuit-hyperplane is crucial for
that property).  To get examples for which both $|P|$ and $|T-P|$ are
arbitrarily large, start with the matroids just constructed and apply
the operation of $t$-expansion, from \cite{gaps}, which, given a
matroid $M$, produces a matroid $M^t$ with similar geometric
structure, but of larger rank and size.  To get $M^t$, where
$t\in \mathbb{N}$ and $t\geq 2$, first extend $M$ to $M'$ by, for each
$e\in E(M)$, adding $t-1$ elements parallel to $e$ if $r(e)=1$, or as
loops if $r(e)=0$; then $M^t$ is the matroid union of $t$ copies of
$M'$.  By \cite[Theorem 3.15]{gaps} (another corollary of Theorem
\ref{thm:char}), a matroid $M$ is a positroid if and only if its
$t$-expansion $M^t$ is a positroid.

We will need the following variation on Lemma \ref{lem:conrestP}.

\begin{lemma}\label{lem:conrestP2}
  Assume that $T$ is independent in both $M$ and $N$, that
  $\emptyset\ne P\subsetneq T$, and that $P$ is a set of clones in
  both $M$ and $N$.  Let $F$ be a connected flat of $B_T(M,N)$ with
  $F\cap T=P$.  If $F\subseteq E(M)$, then the flat $F$ of $M$ is
  connected, and likewise if $F\subseteq E(N)$.  If
  $F\not\subseteq E(M)$ and $F\not\subseteq E(N)$, then the flat $F_M$
  of $M$ is connected, as is the flat $F_N$ of $N$.
\end{lemma}

\begin{figure}
  \centering
  \begin{tikzpicture}[scale=1.4]
   \tikzset{VertexV/.style = {circle, scale=0.8, fill=black!30,
        draw=black, thick}}

    \node[VertexV] (0) at (0:0) {};%
    \node[VertexV] (1) at (0:1) {};%
    \node[VertexV] (2) at (30:1) {};%
    \node[VertexV] (3) at (60:1) {};%
    \node[VertexV] (4) at (90:1) {};%
    \node[VertexV] (5) at (120:1) {};%
    \node[VertexV] (6) at (150:1) {};%
    \node[VertexV] (7) at (180:1) {};%
    \node[VertexV] (8) at (210:1) {};%
    \node[VertexV] (9) at (240:1) {};%
    \node[VertexV] (10) at (270:1) {};%
    \node[VertexV] (11) at (300:1) {};%
    \node[VertexV] (12) at (330:1) {};%

    \foreach \from/\to in
    {1/2,2/3,3/4,4/5,5/6,6/7,7/8,8/9,9/10,10/11,11/12,12/1,0/1,0/4,0/7,0/10}%
    \draw[thick](\from)--(\to);

        \foreach \from/\to in {1/2,2/3,3/4,0/1}%
    \draw[line width=1mm](\from)--(\to);

  \end{tikzpicture}
  \caption{A wheel with some spokes deleted.  The thick edges are the
    rim edges between two consecutive spokes along with one of those
    spokes.}\label{fig:modifiedwhirl}
\end{figure}

\begin{proof}
  The first assertions are immediate since $M$ and $N$ are
  restrictions of $B_T(M,N)$ by Lemma \ref{lem:indeprest}.  Now assume
  that $F\not\subseteq E(M)$ and $F\not\subseteq E(N)$, and, contrary
  to what we must show, that $M|F_M$ is disconnected.  The elements of
  $P$, which are clones in $M$, either are all coloops of $M|F_M$ or
  are all in the same connected component of $M|F_M$.  Now
  $F\not\subseteq E(N)$, so in either case we have
  $M|F_M=(M|X)\oplus (M|(F_M-X))$ for some set $X$ where
  $P\subseteq X\subsetneq F_M$.  Lemma \ref{lem:discflats} gives the
  contradiction that $B_T(M,N)|F$ is disconnected (as $Y$ in that
  lemma, use $F_N$).  Thus, $M|F_M$ is connected.  By symmetry,
  $N|F_N$ is connected.
\end{proof}

\begin{thm}\label{thm:bondmain2}
  Let $M$ and $N$ be positroids with no loops and let $E(M)\cap E(N)$
  be $T$.  Assume that
  \begin{itemize}
  \item[\emph{(1)}] $T$ is independent in both $M$ and $N$,
  \item[\emph{(2)}] $M|\cl_M(T)$ and $N|\cl_N(T)$ are connected, and
  \item[\emph{(3)}] there is a nonempty proper subset $P$ of $T$ for
    which
    \begin{itemize}
    \item[\emph{(3a)}] $P$ is a set of clones in both $M$ and $N$, and
    \item[\emph{(3b)}] if $F$ is a nonsingleton connected flat of
      either $M$ or $N$ with $F\cap (T-P)\ne\emptyset$, then
      $T\subseteq F$.
    \end{itemize}
  \end{itemize}
  If some positroid order for $M$ has some element in $P$ and some
  element in $T-P$ cyclically consecutive, and some positroid order
  for $N$ has some element in $P$ and some element in $T-P$ cyclically
  consecutive, then the free amalgam $B_T(M,N)$ is a positroid.
\end{thm}

\begin{proof}
  Let $P=\{p_1,p_2,\ldots,p_h\}$ and $T-P=\{t_1,t_2,\ldots,t_k\}$.
  Assumption (3b) implies that $T-P$ is a set of clones in both $M$
  and $N$.  By that observation, assumption (3a), Lemma
  \ref{lem:autclone}, and the hypothesis in the last sentence of the
  theorem, $p_h$ and $t_1$ are cyclically consecutive in some
  positroid order for $M$, as are $t_k$ and $p_1$ in some positroid
  order for $N$.  With $P$ and $T-P$ being sets of clones in both $M$
  and $N$, Lemma \ref{lem:autclone}, Corollary \ref{cor:clone}, and
  the observation that the order dual of a positroid order is a
  positroid order, we can assume that the following linear orders are
  positroid orders for $M$ and $N$, respectively:
  $$e_1<_Me_2<_M\cdots<_M e_m<_Mp_1<_Mp_2<_M\cdots<_Mp_h
  <_Mt_1<_Mt_2<_M\cdots <_Mt_k$$ 
  and
  $$t_1<_Nt_2<_N\cdots
  <_Nt_k<_Np_1<_Np_2<_N\cdots<_Np_h<_Nf_1<_Nf_2<_N\cdots<_N f_n.$$ The
  linear order
  $$e_1<e_2<\cdots< e_m<p_1<p_2<\cdots<p_h<f_1<f_2<\cdots<
  f_n<t_1<t_2<\cdots <t_k$$ has both $\leq_M$ and a shift of $\leq_N$
  as the induced orders on $E(M)$ and $E(N)$, respectively.  We claim
  that the linear order $\leq$ on $E(B_T(M,N))$ satisfies the cyclic
  interval property, so $B_T(M,N)$ is a positroid.  Note that
  observations (I$_1$) and (I$_2$) from the proof of Theorem
  \ref{thm:bondmain1} apply to the linear order $<$ just defined on
  $E(B_T(M,N))$.

  Let $F$ be a connected flat of $B_T(M,N)$ with $|F|\geq 2$. If $F$
  contained some but not all elements of $T-P$, then, by assumption
  (3b), each $t_i \in F\cap (T-P)$ would be a coloop of $M|F_M$ and of
  $N|F_N$, so setting $X=Y=\{t_i\}$ in Lemma \ref{lem:discflats} would
  give the contradiction that $B_T(M,N)|F$ is disconnected.  Thus, by
  the hypotheses, there are three options: $F\cap T=\emptyset$,
  $T\subsetneq F$, or $F\cap T=P$.
  
  The cyclic interval property holds when $T\subsetneq F$ by the same
  argument as in the proof of Theorem \ref{thm:bondmain1} once we make
  the following observations.  Since $F$ contains the connected flats
  $\cl_M(T)$ of $M$ and $\cl_N(T)$ of $N$, there are connected
  components $X$ of $M|F_M$ and $Y$ of $N|F_N$ that contain $T$.
  Since $B_T(M,N)|F$, which is $B_T(F_M,F_N)$, is connected, Lemma
  \ref{lem:bds} implies that $F_M=X$ and $F_N=Y$, so $F_M$ is a
  connected flat of $M$, as is $F_N$ in $N$.  We turn to the other two
  cases.
      
  Assume that $F\cap T=\emptyset$.  As before, $F$ is a connected flat
  of either $M$ or $N$, say of $M$.  Thus, $B_T(M,N)/F=B_T(M/F,N)$.
  It follows from Lemma \ref{lem:connflcontr} that assumption (3b)
  also holds for $M/F$.  Also, since $\cl_M(T)$ is connected, no
  element of $T$ is a coloop of $M$, so no element of $T$ is a coloop
  of $M/F$.  Thus, all elements of $T$ are in the same connected
  component, say $X$, of $M/F$.  It follows that any connected
  component $K$ of $B_T(M,N)/F$ with $|K|\geq 2$ falls under one of
  cases (a)--(c) in the proof of Theorem \ref{thm:bondmain1}.  The
  cyclic interval property holds in each case by the argument given in
  that proof.
  
  Now assume that $F\cap T=P$.  We first treat the case with
  $F\not\subseteq E(M)$ and $F\not\subseteq E(N)$, in which case, by
  Lemma \ref{lem:conrestP2}, the restrictions $M|F_M$ and $N|F_N$ are
  connected.  By applying Lemma \ref{lem:contractinP} and then Lemma
  \ref{lem:contrinB}, we have
  \begin{align*}
    B_T(M,N)/F
    & =(B_{T-P}(M/P,N/P))/(F-T)\\
    & =B_{T-P}(M/F_M,N/F_N).
  \end{align*}
  Since no element of $T$ is a coloop of $M$, no element of $T-P$ is a
  coloop of $M/F_M$.  Also, $T-P$ is a set of clones of $M$ and so of
  $M/F_M$.  Therefore, some connected component $X$ of $M/F_M$ has
  $T-P\subseteq X$.  Similarly, some connected component $Y$ of
  $N/F_N$ has $T-P\subseteq Y$.  Thus, any connected component of
  $B_{T-P}(M,N)/F$ is either
  \begin{itemize}
  \item[(a)] a connected component of $M/F_M$ or of $N/F_N$ that is
    disjoint from $T-P$, or
  \item[(b)] a connected component of the bonding
    $B_{T-P}(M/F_M|X,N/F_N|Y)$.
  \end{itemize}
  For case (a), say $K$ is a connected component of $M/F_M$; then any
  cyclic interval in $\leq_M$ that contains $K$ and is disjoint from
  $F_M$ is a cyclic interval in $\leq$, so the cyclic interval
  property holds.  For a connected component $K$ that falls under case
  (b), note that the union of a cyclic interval $I$ in $\leq_M$ for
  which $I\cap T=T-P$ and a cyclic interval $J$ in $\leq_N$ for which
  $J\cap T=T-P$ is a cyclic interval in $\leq$.  Choosing such $I$ and
  $J$ for which $X\subseteq I$, $Y\subseteq J$, and
  $F\cap (I\cup J)=\emptyset$ shows that the cyclic interval property
  holds in case (b).

  Finally, to complete the case with $F\cap T=P$, we can assume, by
  symmetry, that $F\subseteq E(M)$, so $F$ is a connected flat of $M$.
  Now $B_T(M,N)/P=B_{T-P}(M/P,N/P)$ by Lemma \ref{lem:contractinP}, so
  $$B_T(M,N)/F =B_{T-P}(M/F,N/P).$$  As above, all elements of
  $T-P$ are in the same connected component of $M/F$, say $X$.  Thus,
  any connected component $K$ of $B_T(M,N)/F$ falls under one of three
  cases:
  \begin{itemize}
  \item[(a)] $K\subseteq E(N)-P$,
  \item[(b)] $K$ is a connected component of $M/F$ with
    $K\cap (T-P) = \emptyset$, or
  \item[(c)] $K$ is a connected component of $B_{T-P}(M/F|X,N/P|Y)$
    for some set $Y$ for which $T-P\subseteq Y\subseteq E(N)-P$.
  \end{itemize}
  The cyclic interval property holds in case (a) since $E(N)-P$ is a
  cyclic interval in $\leq$.  The argument in case (b) is exactly as
  for case (a) in the previous paragraph.  In case (c), $X$ is a
  subset of a cyclic interval $I$ in $\leq_M$ with $I\cap F=\emptyset$
  and so $I\cap T=T-P$.  Therefore $I\cup (E(N)-P)$ is a cyclic
  interval in $\leq$, and it shows that the cyclic interval property
  holds in this case, thereby completing the proof of the theorem.
\end{proof}

\subsection{Further properties of bonding}
It may be possible to prove other results in the spirit of Theorems
\ref{thm:bondmain1} and \ref{thm:bondmain2}.  In this section, we
treat some results about bonding that may be useful for the further
development of this operation and its applications.

The connected components of bondings play crucial roles above.  Even
if $M$ and $N$ are connected, $B_T(M,N)$ might be disconnected; see
Figure \ref{fig:discbond}.  The next lemma relates the connected
components of $M$ and $N$ to those of $B_T(M,N)$ under certain
hypotheses.

\begin{lemma}\label{lem:compa1}
  Let $K$ be a connected component of $M$. If $T\subseteq K$ and $T$
  is independent in $N$, then $K$ is contained in a connected
  component of $B_T(M,N)$, say $X$, and $X_M=K$.
\end{lemma}

\begin{proof}
  We have $B_T(M,N)|E(M)=M$ by Lemma \ref{lem:indeprest} since $T$ is
  independent in $N$.  Thus, any two elements of $K$ are contained in
  a circuit of $B_T(M,N)$, and so are in the same connected component,
  say $X$, of $B_T(M,N)$.  Lemma \ref{lem:bds} gives $X_M=K$.
\end{proof}

The next lemma treats a partial converse of the first assertion in
Corollary \ref{cor:flatscomefromflats}.  Recall that $(X,Y)$ is a
\emph{modular pair} in a matroid $M$ if
$r(X)+r(Y)=r(X\cup Y)+r(X\cap Y)$.  When $Y$ is independent, this
equality can be rewritten as $r(X\cup Y)=r(X)+|Y-X|$.

\begin{lemma}\label{lem:flatunion}
  Assume that $T$ is independent in both $M$ and $N$.  Fix
  $F\subseteq E(M)\cup E(N)$ where $F_M$ is a flat of $M$ and $F_N$ is
  a flat of $N$.  If $(F_M,T)$ is a modular pair in $M$ and $(F_N,T)$
  is a modular pair in $N$, then $F$ is a flat of $B_T(M,N)$.  In
  particular, if $|T-F|\leq 1$, then $F$ is a flat of $B_T(M,N)$.
\end{lemma}

\begin{proof}
  Recall that $q_i$, for each $i\in[k]$, is added freely to the line
  $\cl_{M\oplus N'}(\{t_i,s_i\})$ of $M\oplus N'$ to form $H$.  We may
  assume that $F\cap T$ is $T_h=\{t_1,t_2,\ldots,t_h\}$, where $h=0$
  and $T_0=\emptyset$ if $F\cap T=\emptyset$.  Set
  $S_h=\{s_1,s_2,\ldots,s_h\}$ and $Q_h=\{q_1,q_2,\ldots,q_h\}$, which
  are empty if $h=0$.  Thus, $F\cup S_h\cup Q_h$ is a flat of $H$.
  Let $F_{N'}=(F\cap E(N'))\cup S_h$. By the two assumed modular
  pairs,
  \begin{align*}
    r_H(F\cup T\cup S)
    & = r_M(F_M\cup T)+r_{N'}(F_{N'}\cup S)\\
    & = r_M(F_M)+|T-T_h|+r_{N'}(F_{N'})+|S-S_h|\\
    & = r_H(F\cup S_h)+2(k-h).
    \end{align*}
  We claim that $F\cup S_h\cup Q$ is a flat of $H$.  Assume instead
  that $F\cup S_h\cup Q$ is not a flat of $H$, so for some
  $e\not \in F\cup S_h\cup Q$, there is a circuit $C$ with
  $e\in C\subseteq F\cup S_h\cup Q\cup e$.  Let
  $I=\{i\,:\,q_i\in C\cap(Q-Q_h)\}$ and $J=[k]-([h]\cup I)$.  Now
  $I\ne\emptyset$ since $F\cup S_h\cup Q_h$ is a flat of $H$.  If
  $i\in I$, then $\{t_i,s_i\}\subseteq \cl_H(C)$, so
  $$\cl_H(F\cup S_h\cup C\cup \{t_j,s_j\,:j\in J\})=\cl_H(F\cup T\cup
  S).$$ The union of $\{q_i\,:\,i\in I\}\cup \{ t_j,s_j\,:j\in J\}$
  and a basis of $F\cup S_h$ spans $F\cup T\cup S$, so
  $r_H(F\cup T\cup S)\leq r_H(F\cup S_h)+|I|+2|J|$.  However,
  $|I|+2|J|<2(k-h)$, contrary to what we just showed.  Thus,
  $F\cup S_h\cup Q$ is a flat of $H$, and so $F$ is a flat of
  $B_T(M,N)$.
\end{proof}

The set $F=\{4,6,9,11\}$ in the example in Figure \ref{fig:bondrev}
shows that the hypothesis about modular pairs in Lemma
\ref{lem:flatunion} is needed.  In the proofs of the corollaries
below, we will use two observations: (1) if $(X,Y)$ is a modular pair
in $M$, then so is $(\cl_M(X),Y)$, and (2) if, in addition, $Y$ is
independent, then $X\cap Y=\cl_M(X)\cap Y$.

\begin{cor}\label{cor:closureB}
  Assume that $T$ is independent in both $M$ and $N$, and
  $X\subseteq E(M)\cup E(N)$. If $(X_M,T)$ is a modular pair in $M$
  and $(X_N,T)$ is a modular pair in $N$, then the closure $\cl_B(X)$
  of $X$ in $B_T(M,N)$ is $\cl_M(X_M)\cup\cl_N(X_N)$.
\end{cor}

\begin{proof}
  We shorten $B_T(M,N)$ to $B$.  Now $\cl_M(X_M)\subseteq \cl_B(X)$
  since $B|E(M)$ is $M$. Likewise, $\cl_N(X_N)\subseteq \cl_B(X)$.
  Thus, $\cl_M(X_M)\cup\cl_N(X_N) \subseteq \cl_B(X)$.  Since
  $(X_M,T)$ is a modular pair in $M$, so is $(\cl_M(X_M),T)$, and
  $\cl_M(X_M)\cap T = X\cap T$.  Likewise, $(\cl_N(X_N),T)$ is a
  modular pair in $N$ and $\cl_N(X_N)\cap T = X\cap T$.  Thus, by
  Lemma \ref{lem:flatunion}, the union $\cl_M(X_M)\cup\cl_N(X_N)$ is a
  flat of $B$.  Since this flat of $B$ contains $X$, we have
  $\cl_B(X) \subseteq \cl_M(X_M)\cup\cl_N(X_N)$.  With the inclusion
  proven above, equality follows.
\end{proof}

\begin{cor}
  Assume that $T$ is independent in both $M$ and $N$, and
  $X\subseteq E(M)\cup E(N)$.  If $(X_M,T)$ is a modular pair in $M$
  and $(X_N,T)$ is a modular pair in $N$, then the rank $r_B(X)$ of
  $X$ in $B_T(M,N)$ is $r_M(X_M)+r_N(X_N)-|X\cap T|$.
\end{cor}

\begin{proof}
  We again shorten $B_T(M,N)$ to $B$.  Corollary \ref{cor:closureB},
  Lemma \ref{lem:flatrankinB}, and the equality
  $\bigl(\cl_M(X_M)\cup\cl_N(X_N)\bigr)\cap T=X\cap T$ give
  \begin{align*}
    r_B(X)
    & =r_B( \cl_B(X))\\
    & = r_B\bigl(\cl_M(X_M)\cup\cl_N(X_N)\bigr)\\
    &= r_M\bigl(\cl_M(X_M)\bigr)+r_N\bigl(\cl_N(X_N)\bigr)-|X\cap T|\\
    &= r_M\bigl(X_M\bigr)+r_N\bigl(X_N\bigr)-|X\cap T|. \qedhere
  \end{align*}
\end{proof}
  
\section{Some excluded minors for the class of positroids}\label{sec:exmin}

As \cite{blum,ARW} show, the class of positroids is minor-closed, that
is, every minor of a positroid is a positroid.  Thus, the class of
positroids is characterized by its excluded minors, which are the
matroids that are not positroids, but all of their proper minors are
positroids.  Blum \cite[Corollary 4.12]{blum} found the excluded
minors for the class of positroids of rank at most three.  In this
section, we use Theorem \ref{thm:char} and its corollaries to identify
infinitely many excluded minors for the class of positroids, focusing
on excluded minors of rank greater than three.  We do not identify all
excluded minors for the class of positroids; indeed, more are given in
Park \cite{Park}.  We start with the truncation to rank $4$ of the
positroid in Example \ref{rank5example}.  (The truncation to rank $3$
is treated in \cite{blum}.)

\begin{figure}
  \centering
  \begin{tikzpicture}[scale=1]
    \draw[thick,gray!60](0,0)--(0,2)--(1.5,3)--(1.5,1)--(0,0)
    --(-1.5,1)--(-1.5,3)--(0,2);%
    \draw[very thick](0,0.3)--(1.3,1.2);%
    \draw[very thick](0,0.3)--(0,1.7);%
    \draw[very thick](0.5,3.3)--(0,1.7);%
    \draw[very thick](0,1)--(-1.3,1.8);%
    \filldraw (0,0.3) node[above left =-1] {\small$a$} circle
    (2.75pt);%
    \filldraw (0,1) node[above right=-1] {\small$b$} circle (2.75pt);%
    \filldraw (1.3,1.2) node[above=1] {\small$t$} circle (2.75pt);%
    \filldraw (0.65,0.75) node[above=1] {\small$s$} circle (2.75pt);%
    \filldraw (0,1.7) node[right=1] {\small$c$} circle (2.75pt);%
    \filldraw (-1.3,1.8) node[above =1] {\small$p$} circle (2.75pt);%
    \filldraw (-0.65,1.4) node[above =1] {\small$q$} circle (2.75pt);%
    \filldraw (0.5,3.3) node[left =1] {\small$v$} circle (2.75pt);%
    \filldraw (0.25,2.5) node[left =1] {\small$u$} circle (2.75pt);%
    \node at (0,-0.5) {$M$};%
  \end{tikzpicture}
  \hspace{1cm}
  \begin{tikzpicture}[scale=1]
    \draw[very thick](0,2)--(0,0)--(2,0);%
    \draw[very thick](0,1)--(2,1);%
    \filldraw (0,0) node[above right =1] {\small$a$} circle (2.75pt);%
    \filldraw (0,1) node[above right=1] {\small$b$} circle (2.75pt);%
    \filldraw (2,0) node[above=1] {\small$t$} circle (2.75pt);%
    \filldraw (1,0) node[above=1] {\small$s$} circle (2.75pt);%
    \filldraw (-0.1,2) node[left=1] {\small$c$} circle (2.75pt);%
    \filldraw (2,1) node[above =1] {\small$p$} circle (2.75pt);%
    \filldraw (1,1) node[above =1] {\small$q$} circle (2.75pt);%
    \filldraw (0.1,2) node[right =1] {\small$v$} circle (2.75pt);%
    \node at (1,-1) {$M/u$};%
  \end{tikzpicture}
    \hspace{1cm}
  \begin{tikzpicture}[scale=1]
    \draw[very thick](1,2)--(0,0)--(-1,2);%
    \filldraw (-0.1,0) node[left =1] {\small$a$} circle (2.75pt);%
    \filldraw (0.1,0) node[right=1] {\small$b$} circle (2.75pt);%
    \filldraw (-1,2) node[left=1] {\small$t$} circle (2.75pt);%
    \filldraw (-0.5,1) node[left=1] {\small$s$} circle (2.75pt);%
    \filldraw (-0.1,2.3) node[left=1] {\small$u$} circle (2.75pt);%
    \filldraw (1,2) node[right =1] {\small$p$} circle (2.75pt);%
    \filldraw (0.5,1) node[right =1] {\small$q$} circle (2.75pt);%
    \filldraw (0.1,2.3) node[right =1] {\small$v$} circle (2.75pt);%
    \node at (0,-1) {$M/c$};%
  \end{tikzpicture}
  \caption{The rank-$4$ truncation of the positroid in Example
    \ref{rank5example} and two of its contractions.}
  \label{fig:ex1cont}
\end{figure}

\begin{example}\label{ex:ex1cont}
  Let $M$ be the rank-$4$ truncation of the positroid in Example
  \ref{rank5example}, labeled as in Figure \ref{fig:ex1cont}.  Let
  $A=\{a,s,t\}$, $B=\{b,p,q\}$, $C=\{c,u,v\}$, and $X=\{a,b,c\}$.  The
  proper connected flats $F$ of $M$ with $|F|\geq 2$ are those four
  sets along with $A\cup X$, $B\cup X$, and $C\cup X$.  For each such
  $F$, the contraction $M/F$ is connected, so by Corollary
  \ref{cor:convinterval}, if $M$ were a positroid, then there would be
  a linear order on $E(M)$ in which each of those flats is a cyclic
  interval.  By Lemma \ref{lem:noco}, there is no such linear order,
  so $M$ is not a positroid.

  The automorphism group of $M$ is transitive on $\{s,t,p,q,u,v\}$ and
  on $\{a,b,c\}$, so in order to show that all proper minors of $M$
  are positroids, it suffices to show that $M\del u$, $M/u$,
  $M\del c$, and $M/c$ are positroids.  We show that Corollary
  \ref{cor:intervalsimpliespositroid} applies to these minors.
  \begin{itemize}[leftmargin=*]
  \item The proper connected flats $F$ of $M\del u$ with $|F|\geq 2$
    are $A$, $B$, $X$, $A\cup X$, and $B\cup X$, which are intervals
    in the linear order $s<t<a<c<b<p<q<v$.
  \item The proper connected flats $F$ of $M\del c$ with $|F|\geq 2$
    are $A$, $B$, and $(C\cup X)-c$, which are intervals in the linear
    order $s<t<a<u<v<b<p<q$.
  \item The proper connected flats $F$ of $M/u$ with $|F|\geq 2$ are
    $A$, $B$, $C-u$, and $X\cup v$, which are intervals in the
    linear order $s<t<a<c<v<b<p<q$.
  \item The proper connected flats $F$ of $M/c$ with $|F|\geq 2$ are
    $A\cup b$, $B\cup a$, $C-c$, and $X-c$, which are intervals in the
    linear order $s<t<a<b<p<q<u<v$. 
  \end{itemize}
  Thus, $M$ is an excluded minor for the class of positroids.
  \hfill$\circ$
\end{example}

\begin{figure}
  \centering
  \begin{tikzpicture}[scale=0.55]
    \draw[thick, black!30](16,-0.5)--(20,0.5)--(20,6.25)--(16,5.25)--(12,6.25)
    --(12,0.5)--(16,-0.5);%
    \draw[thick, black!30](16,-0.5)--(16,5.25);%

    \draw[very thick](0,2)--(3,0)--(3,4)--(0,2);%
    \filldraw (3,2) node[right=1] {\small$f$} circle (5pt);%
    \filldraw (3,0) node[right=1] {\small$e$} circle (5pt);%
    \filldraw (1.5,1) node[below left] {\small$d$} circle (5pt);%
    \filldraw (0,2) node[left=1] {\small$c$} circle (5pt);%
    \filldraw (1.5,3) node[above left] {\small$b$} circle (5pt);%
    \filldraw (3,4) node[right=1] {\small$a$} circle (5pt);%

    \draw[very thick](6,4)--(9,2)--(6,0);%
    \filldraw (6,0) node[left =1] {\small$e$} circle (5pt);%
    \filldraw (7.5,1) node[below right] {\small$g$} circle (5pt);%
    \filldraw (9,2) node[right =1] {\small$h$} circle (5pt);%
    \filldraw (7.5,3) node[above right] {\small$i$} circle (5pt);%
    \filldraw (6,4) node[left =1] {\small$a$} circle (5pt);%

    \draw[very thick](16,0.5)--(16,4.5);%
    \draw[very thick](16,0.5)--(19,2.5)--(16,4.5);%
    \draw[very thick](16,0.5)--(13,2.5)--(16,4.5);%
    
    \filldraw (16,0.5) node[below left =-1] {\small$e$} circle (5pt);%
    \filldraw (16,2.1) node[right =1] {\small$f$} circle (5pt);%
    \filldraw (16,4.5) node[above left =-1] {\small$a$} circle (5pt);%
    
    \filldraw (14.5,1.5) node[below left] {\small$d$} circle (5pt);%
    \filldraw (14.5,3.5) node[above left] {\small$b$} circle (5pt);%
    \filldraw (13,2.5) node[left =1] {\small$c$} circle (5pt);%
    \filldraw (17.5,1.5) node[below right] {\small$g$} circle (5pt);%
    \filldraw (19,2.5) node[right =1] {\small$h$} circle (5pt);%
    \filldraw (17.5,3.5) node[above right] {\small$i$} circle (5pt);%

    \node at (1.5,-1.5) {$M$};%
    \node at (7.5,-1.5) {$N$};%
    \node at (16,-1.5) {$B_T(M,N)$};%
  \end{tikzpicture}
  \caption{Two rank-$3$ positroids $M$ and $N$, and their free amalgam
    $B_T(M,N)$, which is a rank-$4$ excluded minor for the class of
    positroids.}
  \label{fig:bondrevvar2}
\end{figure}

Part of the interest in the next example is to show that the free
amalgam of positroids, when the hypotheses of neither Theorem
\ref{thm:bondmain1} nor Theorem \ref{thm:bondmain2} hold, need not
yield a positroid; indeed, it may produce an excluded minor for the
class of positroids.

\begin{example}
  Consider the free amalgam $B=B_T(M,N)$ of the matroids $M$ and $N$
  shown in Figure \ref{fig:bondrevvar2}.  There is no positroid order
  for $B$ since, if there were, then by Corollary
  \ref{cor:convinterval}, the sets $\{a,b,c\}$, $\{c,d,e\}$,
  $\{e,g,h\}$, and $\{h,i,a\}$ would have to be cyclic intervals in
  this linear order, but that is impossible since $f$ would have to be
  in one of those cyclic intervals.  Thus, while $M$ and $N$ are
  positroids, their free amalgam is not.  To show that $B$ is an
  excluded minor for the class of positroids, by the symmetry of $B$,
  it suffices to show that $B\del x$ and $B/x$ are positroids for each
  $x\in\{a,b,c,f\}$.

  Each contraction $B/x$ with $x\in\{a,b,c,f\}$ is a parallel
  extension of a matroid that either has at most two proper nonempty
  cyclic flats or satisfies the hypothesis of Corollary
  \ref{cor:3cyc}, so $B/x$ is a positroid by Corollary
  \ref{cor:parcon}.  The deletion $B\del f$ is the whirl
  $\mathcal{W}^4$, which is a multi-path matroid, and $B\del b$ is a
  parallel connection of $\mathcal{W}^3$ and $U_{2,3}$, and so is a
  positroid. The deletion $B\del a$ is a positroid since its four
  proper connected flats $F$ with $|F|\geq 2$, namely, $\{c,d,e\}$,
  $\{e,g,h\}$, $\{f,b,c,d,e\}$, $\{e,g,h,i,f\}$, are cyclic intervals
  in the linear order $b<c<d<e<g<h<i<f$.  Lastly, for $B\del c$, the
  proper connected flats $\{h,i,a\}$, $\{e,g,h\}$, $\{a,b,d,f,e\}$,
  and $\{f,e, g,h,i,a\}$ are cyclic intervals in the linear order
  $a<b<d<f<e<g<h<i$, and for $F=\{a,f,e\}$, the only proper connected
  flat with $|F|\geq 2$ that is not a cyclic interval, the two
  connected components of $(B\del c)/F$, both of rank $1$, are
  $\{b,d\}$ and $\{g,h,i\}$, which, as needed, are cyclic intervals.
  \hfill$\circ$
\end{example}

Examples \ref{example:genK4}--\ref{ex:last} give infinite families of
excluded minors for the class of positroids, most with multiple
parameters.  Example \ref{example:genK4} includes $M(K_4)$ and Example
\ref{ex:freeextwhirlgen} includes free extensions of rank-$3$
truncations of whirls.

\begin{example}\label{example:genK4}
  Let $\mathcal{L}$ be
  $\bigl\{\{1,3,5\},\,\{2,3,4\},\,\{4,5,6\},\,\{1,2,6\}\bigr\}$, the
  set of $3$-point lines of the cycle matroid $M(K_4)$ as labeled in
  Figure \ref{fig:whirl}.  Let $X_1,X_2,\ldots,X_6$ be pairwise
  disjoint nonempty sets that satisfy the following conditions.  Let
  $x_i=|X_i|$.  The set $X_2$ is arbitrary.  Pick $X_1$, $X_3$, $X_4$,
  and $X_6$ so that $x_3+x_4=x_1+x_6$.  Let $r$ be $x_2+x_3+x_4$,
  which is also $x_1+x_2+x_6$.  We may assume that
  $x_1+x_3\leq x_4+x_6$.  Choose $X_5$ so that
  $x_1+x_3+x_5\leq r\leq x_4+x_5+x_6$.  Let $E=X_1\cup\cdots\cup X_6$,
  and let $\mathcal{Z}$ consist of $\emptyset$, $E$, and the four sets
  $X_i\cup X_j\cup X_k$ for $\{i,j,k\}\in\mathcal{L}$, and let the
  ranks of these sets be as given in the table below.
  \begin{center}
    \begin{tabular}{|c|c|c|c|c|c|}
      \hline
      $\emptyset$
      & $X_1\cup X_3\cup X_5$
      & $X_2\cup X_3\cup X_4$
      & $X_4\cup X_5\cup X_6$
      & $X_1\cup X_2\cup X_6$
      & $E$\\ \hline 
      $0$ &$x_1+x_3+x_5-1$ &$r-1$&$r-1$&$r-1$& $r$\\\hline
    \end{tabular}
  \end{center}
  Properties (Z0) and (Z1) in Theorem \ref{thm:cfaxioms} clearly hold.
  The only case of property (Z2) that is not immediate is when
  $X=X_1\cup X_3\cup X_5$ and $Y=E$, in which case we need
  $r(Y)-r(X)<|Y-X|$, that is, $r-(x_1+x_3+x_5-1)<x_2+x_4+x_6$, that
  is, $r+1<|E|$, which is clear.  Thus, property (Z2) holds.  Each
  instance of the inequality in property (Z3) for two incomparable
  cyclic flats other than $X_1\cup X_3\cup X_5$ is an inequality of
  the form $r+x_i\leq 2(r-1)$, i.e., $x_i\leq r-2$, which holds since
  the sum of $x_i$ and two other positive integers is at most $r$. The
  inequality required for property (Z3) in each case that involves
  $X_1\cup X_3\cup X_5$ has the form $r+x_i\leq r-1+x_1+x_3+x_5-1$
  with $i\in\{1,3,5\}$, i.e., $x_i\leq x_1+x_3+x_5-2$, which holds
  since $x_1,x_3,x_5\in\mathbb{N}$.  Thus, the cyclic flats and ranks
  given above define a matroid $M$ on $E$.

  We now show that $M$ is an excluded minor for the class of
  positroids.  By Lemma \ref{lem:cyclicminor}, the restriction to and
  contraction by $X_i\cup X_j\cup X_k$, for each
  $\{i,j,k\}\in\mathcal{L}$, is a uniform matroid of positive nullity,
  and so is connected.  If $M$ were a positroid, then each of these
  connected flats would be a cyclic interval in any positroid order
  for $M$; that is impossible by Lemma \ref{lem:noco}, so $M$ is not a
  positroid.  Each cyclic flat $X_i\cup X_j\cup X_k$ for
  $\{i,j,k\}\in\mathcal{L}-\{\{4,5,6\}\}$ is a circuit
  ($X_4\cup X_5\cup X_6$ might be a circuit) and each cyclic flat
  $X_i\cup X_j\cup X_k$ for $\{i,j,k\}\in\mathcal{L}-\{\{1,3,5\}\}$ is
  a hyperplane ($X_1\cup X_3\cup X_5$ might be a hyperplane).  Thus,
  each element $e$ is in either one or two cyclic flats that are
  circuits, and not in either one or two cyclic flats that are
  hyperplanes.  So $M\del e$ either has only four cyclic flats or the
  hypotheses of Corollary \ref{cor:3cyc} apply, and likewise for
  $M/e$, so these minors of $M$ are positroids.  Thus, $M$ is an
  excluded minor for the class of positroids.  \hfill$\circ$
\end{example} 

In Example \ref{example:genK4}, the proper, nonempty cyclic flat (if
any) that is not a circuit and the one (if any) that is not a
hyperplane are not equal. In the next example, they are the same.

\begin{example}\label{example:genK4var1}
  Fix $a,b,c,s,r\in \mathbb{N}$ for which (i) $\max(a,b,c)<s<a+b+c$,
  (ii) $r$ has the same parity as $a+b+c$, and (iii)
  $\max(s,a+b-c,a+c-b,b+c-a)<r$.  Let $\mathcal{L}$ be as in Example
  \ref{example:genK4}. Pick pairwise disjoint sets
  $X_1,X_2,\ldots,X_6$ for which $|X_4|=a$, $|X_5|=c$, $|X_6|=b$, and
  $$|X_1|=\frac{r+a-b-c}{2}, \qquad |X_2|=\frac{r-a-b+c}{2} \qquad
  |X_3|=\frac{r-a+b-c}{2}.$$ Thus, $|X_4\cup X_5\cup X_6|=a+b+c$ and
  $|X_i\cup X_j\cup X_k|=r$, for
  $\{i,j,k\}\in\mathcal{L}-\{\{4,5,6\}\}$.  Let
  $E=X_1\cup\cdots\cup X_6$, and let $\mathcal{Z}$ consist of
  $\emptyset$, $E$, and the four sets $X_i\cup X_j\cup X_k$ for
  $\{i,j,k\}\in\mathcal{L}$, and let the ranks of these sets be as
  given in the table below.
  \begin{center}
    \begin{tabular}{|c|c|c|c|c|c|}
      \hline
      $\emptyset$
      & $X_1\cup X_3\cup X_5$
      & $X_2\cup X_3\cup X_4$
      &  $X_1\cup X_2\cup X_6$
      &  $X_4\cup X_5\cup X_6$
      & $E$\\ \hline 
      $0$ &$r-1$ &$r-1$&$r-1$&$s$& $r$\\\hline
    \end{tabular}
  \end{center}
  Properties (Z0) and (Z1) in Theorem \ref{thm:cfaxioms} clearly hold.
  The only case of property (Z2) that is not transparent is when
  $X=X_4\cup X_5\cup X_6$ and $Y=E$, for which we must show that
  $r(Y)-r(X)<|Y-X|$, that is, $r-s<|X_1\cup X_2\cup X_3|$.  That
  inequality simplifies to $a+b+c<r+2s$, which holds since
  $a+b+c< 3s<r+2s$ by conditions (i) and (iii).  Thus, property (Z2)
  holds.  Each instance of the inequality in property (Z3) for two
  incomparable cyclic flats other than $X_4\cup X_5\cup X_6$ is an
  inequality of the form $r+|X_i|\leq 2(r-1)$, i.e., $|X_i|\leq r-2$,
  which holds since $|X_i|$ and two other positive integers add to
  $r$. The inequality required for property (Z3) in each case that
  involves $X_4\cup X_5\cup X_6$ has the form $r+|X_i|\leq r-1+s$ with
  $i\in\{4,5,6\}$, i.e., $|X_i|<s$, which holds by the inequality
  $\max(a,b,c)<s$ in condition (i).  Thus, the cyclic flats and their
  ranks define a matroid $M$ on $E$.  The same argument as used in the
  last example shows that $M$ is an excluded minor for the class of
  positroids.  \hfill$\circ$
\end{example}

We next give three more infinite families of excluded minors for the
class of positroids, along with their duals, that are different
variations on the idea in Example \ref{example:genK4}.  We omit the
proofs that these are excluded minors since no new ideas are required.

\begin{example}
  Fix $a,b,c,k\in \mathbb{N}$ and let $X_1,\ldots,X_6$ be pairwise
  disjoint sets with $|X_3|=a$, $|X_6|=a+k$,  $|X_4|=b$, $|X_1|=b+k$,
  $|X_5|=c$, and $|X_2|=c+k$.  Fix an element $p$ that is in none of
  $X_1,\ldots,X_6$, and let $E=X_1\cup\cdots\cup X_6\cup p$.  Let the
  paving matroid $M$ of rank $a+b+c+k+1$ on $E$ have as its dependent
  hyperplanes the three circuit-hyperplanes
  $X_1\cup X_3\cup X_5\cup p$, $X_2\cup X_3\cup X_4\cup p$, and
  $X_4\cup X_5\cup X_6\cup p$, along with $X_1\cup X_2\cup X_6$, which
  has nullity $2k$.  It is easy to check that $M$ is an excluded minor
  for the class of positroids.  In the dual, $M^*$, of $M$, which has
  rank $a+b+c+2k$, the sets $X_2\cup X_4\cup X_6$,
  $X_1\cup X_5\cup X_6$, and $X_1\cup X_2\cup X_3$ are
  circuit-hyperplanes, while $X_3\cup X_4\cup X_5\cup p$ is a circuit
  but not a hyperplane.  \hfill$\circ$
\end{example}

\begin{example}
  Fix $a,b,c\in \mathbb{N}$ and let the sets $X_1,\ldots,X_6$ be
  pairwise disjoint and satisfy $|X_3|= |X_6|=a$, $|X_1|=b+1$,
  $|X_4|=b$, and $|X_2|=|X_5|=c$.  Fix elements $p$ and $q$ that are
  in none of $X_1,\ldots,X_6$, and let
  $E=X_1\cup\cdots\cup X_6\cup \{p,q\}$.  The sets
  $$X_1\cup X_3\cup X_5\cup p, \,\,
  X_2\cup X_3\cup X_4\cup \{p,q\}, \,\, X_4\cup X_5\cup X_6\cup
  \{p,q\},\,\, X_1\cup X_2\cup X_6\cup q$$ are the circuit-hyperplanes
  of a sparse paving matroid $M$ of rank $a+b+c+2$ on $E$.  It is easy
  to check that $M$ is an excluded minor for the class of positroids.
  The sparse paving matroid $M^*$ has rank $a+b+c+1$ and its
  circuit-hyperplanes are $X_2\cup X_4\cup X_6\cup q$,
  $X_1\cup X_5\cup X_6$, $X_1\cup X_2\cup X_3$, and
  $X_3\cup X_4\cup X_5\cup p$.  \hfill$\circ$
\end{example}

\begin{example}\label{ex:K4last}
  Fix $a,b,c\in \mathbb{N}$ and let the sets $X_1,\ldots,X_6$ be
  pairwise disjoint and satisfy $|X_3|= |X_6|=a$, $|X_1|=|X_4|=b$, and
  $|X_2|=|X_5|=c$.  Fix four elements $p,q,s,t$ that are in none of
  $X_1,\ldots,X_6$, and let $E=X_1\cup\cdots\cup X_6\cup \{p,q,s,t\}$.
  The sets
  $$X_1\cup X_3\cup X_5\cup \{q,s,t\}, \qquad X_2\cup X_3\cup X_4\cup
  \{p,s,t\},$$
  $$X_4\cup X_5\cup X_6\cup \{p,q,t\}, \qquad
  X_1\cup X_2\cup X_6\cup \{p,q,s\}$$ are the circuit-hyperplanes of a
  sparse-paving matroid $M$ on $E$ of rank $a+b+c+3$.  It is easy to
  check that $M$ is an excluded minor for the class of positroids.
  The sparse-paving matroid $M^*$ has the circuit-hyperplanes
  $X_2\cup X_4\cup X_6\cup p$, $X_1\cup X_5\cup X_6\cup q$,
  $X_1\cup X_2\cup X_3\cup s$, and $X_3\cup X_4\cup X_5\cup t$, and
  its rank is $a+b+c+1$.  \hfill$\circ$
\end{example}

Let $L$ be the lattice of flats of $M(K_4)$ and let $M$ be a loopless
matroid whose lattice of cyclic flats is isomorphic to $L$.  Note that
if the cyclic flats of a matroid $N$ are either (i) just $\emptyset$
and $E(N)$, or (ii) just $\emptyset$, $E(N)$, and three other sets,
none of which contains either of the other two, then $N$ is connected
since its lattice of cyclic flats is not a direct product of two other
lattices.  By Lemma \ref{lem:cyclicminor}, for any cyclic flat $F_1$
of $M$ that corresponds to a point of $M(K_4)$ and any cyclic flat
$F_2$ of $M$ that corresponds to a $3$-point line of $M(K_4)$,
condition (i) holds for $\mathcal{Z}(M|F_1)$ and $\mathcal{Z}(M/F_2)$
while condition (ii) holds for $\mathcal{Z}(M/F_1)$ and
$\mathcal{Z}(M|F_2)$, so all such minors are connected.  As above, it
follows that $M$ is not a positroid.  Thus, while every lattice is
isomorphic to the lattice of cyclic flats of some matroid
\cite{cycflats}, the same is not true for positroids.

The rank-$n$ whirl, $\mathcal{W}^n$, is a multi-path matroid, so it and
its rank-$3$ truncation are positroids, but the free extension of the
rank-$3$ truncation of $\mathcal{W}^n$ is an excluded minor for the
class of positroids.  The next example treats these excluded minors
and many more.

\begin{example}\label{ex:freeextwhirlgen}
  Fix $n,r,x_1,x_2,\ldots,x_{2n}\in\mathbb{N}$, with $n\geq 3$, that
  satisfy the properties below, where we interpret indices modulo $2n$
  and set $m_i=x_{2i-2}+x_{2i-1}+x_{2i}$ for all $i\in[n]$:
  \begin{itemize}
  \item[(i)] $3\leq m_i\leq r$ for all $i\in [n]$,
  \item[(ii)] if $n=3$, then all $m_i$ are $r$, while if $n>3$, then
    $m_i=r=m_j$ for at least two elements $i$ and $j$ of $[n]$ that
    are not cyclically consecutive, 
  \item[(iii)] if $i,j\in[n]$ with $j\ne i$ and $j\ne i+1$,
     then $\displaystyle r<x_{2i}+x_{2j}+\sum_{k=i+1}^jx_{2k-1}$.
  \end{itemize}
  (To be clear, the sum in property (iii) is over the odd terms in a
  cyclic interval in $[2n]$.)  Let $e_1,e_2,\ldots,e_{2n}$ be the
  elements of the $n$-whirl $\mathcal{W}^n$, labeled so that, for each
  $i\in[n]$, the set $\{e_{2i-2},e_{2i-1},e_{2i}\}$ is a $3$-circuit,
  where $e_0=e_{2n}$.  For each $i\in[2n]$, apply series extension
  $x_i-1$ times to $e_i$, so the $3$-circuit
  $\{e_{2i-2},e_{2i-1},e_{2i}\}$ of $\mathcal{W}^n$ yields an
  $m_i$-circuit $C_i$.  Let $M'$ denote the resulting matroid.
  Truncate the free extension $M'+f$ to rank $r$ to get a matroid $M$.
  We claim that $M$ is an excluded minor for the class of positroids.

  The circuits of the $n$-whirl $\mathcal{W}^n$ are the symmetric
  differences of consecutive $3$-circuits, so the circuits of $M'$ are
  the symmetric differences of the form
  $C_i\triangle C_{i+1}\triangle\cdots\triangle C_j$, where we
  interpret the subscripts modulo $n$.  Thus, by property (iii), the
  only circuits of $M$ of rank less than $r$ are the circuits $C_i$
  for $i\in [n]$, so they are the only proper connected flats $F$ of
  $M$ with $|F|\geq 2$.  Also, $M/C_i$ is connected by part (3) of
  Lemma \ref{lem:conncond}.  So if $M$ were a positroid, each circuit
  $C_i$ would have to be an interval in any positroid order for $M$,
  but, since $f$ is in no circuit $C_i$ and
  $C_i\cap C_{i+1}\ne\emptyset$ for all $i\in[n]$, there is no
  positroid order for $M$, so $M$ is not a positroid.

  For each $e\in E(M)$, at least one circuit $C_i$ of $M$ fails to be
  a circuit of $M/e$ by property (ii), so a now-routine use of
  Corollary \ref{cor:intervalsimpliespositroid} implies that $M/e$ is
  a positroid.  The same holds for $M\del e$ if $e\ne f$.  Adapting a
  positroid order for a whirl to $M\del f$ shows that it too is a
  positroid.  Thus, $M$ is an excluded minor for the class of
  positroids.  \hfill$\circ$
\end{example}

\begin{figure}
  \centering
  \begin{tikzpicture}[scale=1] 
      \draw[very thick](6,0)--(8,-1)--(8,1)--(6,0);%
      \draw[very thick](7.2,-0.6)--(7.2,0.6);%

      \foreach \point in
      {(8,0),(8,1),(8,-1),(7.2,-0.6),(7.2,0.6),(6,0),(7.2,0)}
      \filldraw \point circle (2.75pt);%
    \end{tikzpicture}
    \hspace{2cm}
    \begin{tikzpicture}[scale=0.6]      
      \filldraw[black!20] (210:1)--(210:3.3)--
      (330:3.3)--(330:1)--(210:1);%
      \filldraw[black!13] (210:1)--(210:3.3)--
      (90:3.3)--(90:1)--(210:1);%
      \filldraw[black!5] (330:1)--(330:3.3)--
      (90:3.3)--(90:1)--(330:1);%

      \draw[very thin] (210:1)--(210:3.3)--(330:3.3)--(330:1)--(210:1);%
      \draw[very thin] (210:1)--(210:3.3)--(90:3.3)--(90:1)--(210:1);%
      \draw[very thin] (330:1)--(330:3.3)--(90:3.3)--(90:1)--(330:1);%

      \draw[very thick] (40:0.8)--(140:0.8);%
      \draw[very thick] (190:1)--(350:1);%
      
      \filldraw (90:2.2) circle (4pt);%
      \filldraw (40:0.8) circle (4pt);%
      \filldraw (140:0.8) circle (4pt);%
      \filldraw (350:1) circle (4pt);%
      \filldraw (190:1) circle (4pt);%
      \filldraw (340:2) circle (4pt);%
      \filldraw (200:2) circle (4pt);%
      \filldraw (90:0.5) circle (4pt);%
      \filldraw (270:0.2) circle (4pt);%
    \end{tikzpicture}
    \caption{The matroids in Example \ref{ex:whirlvar} for ranks $3$
      and $4$, with the $4$-circuits in the rank-$4$ matroid drawn in
      the faces of a triangular prism.}
  \label{fig:P7etc}
\end{figure}

The next infinite family of excluded minors for the class of
positroids can be seen as a variation on a $3$-whirl, but with two
$r$-circuits and a $3$-circuit, and instead of taking a free
extension, another $3$-circuit is added.

\begin{example}\label{ex:whirlvar}
  Fix an integer $r\geq 3$, take the parallel connection of two
  $r$-circuits, $A$ and $B$, with base point $e$, truncate to rank
  $r$, fix $a_1,a_2\in A-e$ and $b_1,b_2\in B-e$, and then, using
  principal extension, add a point $p_1$ to the line $\{a_1,b_1\}$ and
  a point $p_2$ to the line $\{a_2,b_2\}$.  The resulting matroid $M$
  has rank $r$ and its proper, nonempty connected flats, each of which
  is a circuit, are $A$, $B$, $\{a_1,b_1,p_1\}$, and
  $\{a_2,b_2,p_2\}$.  (See Figure \ref{fig:P7etc} for the cases $r=3$
  and $r=4$.)  Contracting any of the proper, nonempty connected flats
  yields a connected matroid, and, by Lemma \ref{lem:noco}, no linear
  order on $E(M)$ has each of $A$, $B$, $\{a_1,b_1,p_1\}$, and
  $\{a_2,b_2,p_2\}$ being a cyclic interval, so $M$ is not a
  positroid.  It is routine to check that all proper minors of $M$ are
  positroids.  \hfill$\circ$
\end{example}

We close with two more infinite families of excluded minors for the
class of positroids.  With the ideas above, verifying that these are
excluded minors is routine.  Also, applying duality yields more
excluded minors.

\begin{example}
  For integers $n\geq k\geq 3$, take the parallel connection of two
  $n$-circuits and a $k$-circuit (or three $n$-circuits if $n=k$) at a
  common base point, and then truncate to rank $n$.  The matroid
  obtained is an excluded minor for the class of positroids.
\end{example}

\begin{example}\label{ex:last}
  For integers $n\geq k\geq 3$, take two $n$-circuits and a
  $k$-circuit (or three $n$-circuits if $n=k$) and a $3$-point line;
  for each of the first three circuits, take its parallel connection
  with the line, using three different base points on the line, and
  then truncate to rank $n$.  The resulting matroid is an excluded
  minor for the class of positroids.
\end{example}

\vspace{5pt}

\begin{center}
 \textsc{Acknowledgments}
\end{center}

\vspace{3pt}

The author thanks Kolja Knauer for pointing out the work of Blum and
the result of Lam and Postnikov that base-sortable matroids are the
same as positroids.  The author thanks Felipe Rinc\'{o}n for pointing
him to Theorems \ref{thm:arw2} and \ref{thm:pchar} and recasting those
results, which were originally stated in terms of matroid polytopes,
in a form closer to what is given here, and for a useful discussion
about Blum's conjecture that led to Theorem \ref{thm:equivchars}.  The
author also thanks both referees for their careful reading and helpful
comments that should result in making this paper accessible to a wider
audience.

\end{document}